\documentclass[leqno]{article}
\usepackage{palatino} % This appears to work best with the Euler math fonts
\usepackage{verbatim,amsmath,amsfonts,amssymb,array,theorem, subfigure}
\usepackage[mathbf,mathcal]{euler}
\usepackage{pifont,fancyheadings_ssv,wrapfig,graphicx}
\usepackage{epstopdf}
\usepackage[usenames,dvipsnames]{xcolor}
\usepackage{tikz, tikz-cd}
\usepackage{extarrows}
\usepackage{enumitem}
\usepackage{mathtools}
\usepackage[top=1in, bottom=1in, outer=1in, inner=1in, heightrounded, marginparwidth=1.5cm, marginparsep=1cm]{geometry}
 \usepackage[savepos]{zref}% http://ctan.org/pkg/zref
\usetikzlibrary{matrix, calc, positioning}
\usepackage{thmtools}
\usepackage{thm-restate}
\usepackage{pgfplots}
\usepackage{hyperref}
\usepackage{bm}
\usepackage{dsfont}
\usepackage{comment}

\definecolor{amaranth}{rgb}{0.9, 0.17, 0.31}

\renewcommand{\emptyset}{\mathop{\varnothing}} % Redefines \emptyset

\newcommand{\qed}{\mbox{\hfill$\square$}}

\newcommand{\bigslant}[2]{{\raisebox{.2em}{$#1$}\left/\raisebox{-.2em}{$#2$}\right.}}

\newcommand{\im}{\textnormal{im}}
\newcommand{\ZZ}{\mathbb{Z}}

\newcommand{\RR}{\mathbb{R}}

\newcommand{\CC}{\mathbb{C}}
\newcommand{\NN}{\mathbb{N}}

\newcommand{\dd}{\partial}

\newcommand{\la}{\langle}
\newcommand{\ra}{\rangle}
\newcommand{\cg}[1]{{\cal #1}}
\newcommand{\Int}{\textnormal{Int}}

\newcounter{acount}
\newcounter{Example}
\newenvironment{example}%
    {\renewcommand{\theExample}{\arabic{Example}}\refstepcounter{Example}%
     \begin{list}{}%
        {\usecounter{Examples}%
         \setlength{\labelsep}{0pt}\setlength{\leftmargin}{0pt}%
         \setlength{\labelwidth}{0pt}\setlength{\listparindent}{0.5in}%
        }%
     \item[\normalfont\textsc{Example} \arabic{Example}) ]%
     \renewcommand{\theExample}{\arabic{Example}}%
    }
    {\end{list}}
\newcounter{Examples}
\newcommand{\firstexample}%
    {\renewcommand{\Lbl}{\textsc{Examples: }}%
     \setcounter{Examples}{\theExample}%
    }

    {\renewcommand{\theExample}{\arabic{Example}}%
     \setcounter{Examples}{\theExample}%
     \refstepcounter{Examples}%
     \begin{list}{\Lbl\arabic{Examples}) }%
        {\usecounter{Examples}%
         \setlength{\labelsep}{0pt}\setlength{\leftmargin}{0pt}%
         \setlength{\labelwidth}{0pt}\setlength{\listparindent}{0.5in}%
        }%
    }%
    {\setcounter{Example}{\theExamples}%
     \renewcommand{\theExample}{\arabic{Example}}\end{list}%
    }   
\newcommand{\Lbl}{}                            % Label identifies exercise type
\newenvironment{proof}{%        
    \medskip\noindent\textsc{Proof}:
  }{
    \hfill\qed\medskip
  }
\newenvironment{proofprop1}{%        
    \medskip\noindent\textsc{Proof of Proposition 1}:
  }{
    \hfill\qed\medskip
  }  
\newenvironment{proofthm1}{%        
    \medskip\noindent\textsc{Proof of Theorem 1}:
  }{
    \hfill\qed\medskip
  }  
\newenvironment{proofthm8}{%        
    \medskip\noindent\textsc{Proof of Theorem 8}:
  }{
    \hfill\qed\medskip
  }  
\newenvironment{proofprop5}{%        
    \medskip\noindent\textsc{Proof of Proposition 5}:
  }{
    \hfill\qed\medskip
  }  
\newenvironment{proofprop4}{%        
    \medskip\noindent\textsc{Proof of Proposition 4}:
  }{
    \hfill\qed\medskip
  }  
\newenvironment{proofprop6}{%        
    \medskip\noindent\textsc{Proof of Proposition 6}:
  }{
    \hfill\qed\medskip
  }  
\newcounter{bbean}
  \newcounter{Remarks}
  \newenvironment{remark}%
    {\renewcommand{\theRemark}{\arabic{Remark}}\refstepcounter{Remark}%
     \begin{list}{}%
        {\usecounter{Remarks}%
         \setlength{\labelsep}{0pt}\setlength{\leftmargin}{0pt}%
         \setlength{\labelwidth}{0pt}\setlength{\listparindent}{0.5in}%
        }%
     \item[\normalfont\textsc{Remark} \arabic{Remark}) ]%
     \renewcommand{\theRemark}{\arabic{Remark}}%
    }
    {\end{list}}
  \newcounter{Remark}
\newcommand{\firstremark}%
    {\renewcommand{\Lbl}{\textsc{Remarks: }}%
     \setcounter{Remarks}{\theRemark}%
    }

    {\renewcommand{\theRemark}{\arabic{Remark}}%
     \setcounter{Remarks}{\theRemark}%
     \refstepcounter{Remarks}%
     \begin{list}{\Lbl\arabic{Remarks}) }%
        {\usecounter{Remarks}%
         \setlength{\labelsep}{0pt}\setlength{\leftmargin}{0pt}%
         \setlength{\labelwidth}{0pt}\setlength{\listparindent}{0.5in}%
        }%
    }%
    {\setcounter{Remark}{\theRemarks}%
     \renewcommand{\theRemark}{\arabic{Remark}}\end{list}%
    }   
\theoremstyle{plain}
\newtheorem{conjecture}{Conjecture}            %	Conjecture
\newtheorem{corollary}{Corollary}              %	Corollary
{\theorembodyfont{\rmfamily}
}
{\theorembodyfont{\rmfamily} }
\newtheorem{lemma}{Lemma}                      %	Lemma
\newtheorem{proposition}{Proposition}          %	Proposition
\newtheorem{theorem}{Theorem}                  %	Theorem

{\theorembodyfont{\rmfamily} }
{\theorembodyfont{\rmfamily} }
{\theorembodyfont{\rmfamily} }
\theoremheaderfont{\scshape}
 \hypersetup{
    colorlinks,
    citecolor=blue,
    filecolor=blue,
    linkcolor=blue,
    urlcolor=blue
}

\tikzcdset{arrow style=tikz}

\title{Periodic leaf-wise intersection points from Lagrangians}
\date{}
\author{Sara Venkatesh}

\begin{document}
\maketitle
\abstract{We investigate leaf-wise intersection points on hypersurfaces of contact type in monotone symplectic manifolds.  We show that monotone Floer-essential Lagrangians detect periodic leaf-wise intersection points in hypersurfaces of contact type whose Reeb flow is Zoll.  Examples include the prequantization bundles appearing in monotone toric negative line bundles.  Generalizing, we prove the existence of leaf-wise intersection points for certain annulus subbundles in weak\textsuperscript{+}-monotone negative line bundles, not necessarily toric.  The proofs combine reduced symplectic cohomology with the original methods employed by Albers-Frauenfelder to prove global existence results of this kind.}
\tableofcontents

\section{Introduction}

The concept of a {\it leaf-wise intersection point} was first introduced by Moser in his 1978 paper \cite{moser}, where he studied certain types of fixed points that occur when one perturbs a dynamical system.  Imagine a particle moving along a trajectory, as prescribed by some dynamical system, when suddenly a perturbation disturbs the system: the particle is thrown off course.  After some finite time, the perturbation ceases and the dynamical system returns to its original form.  Is there any chance that, miraculously, the particle now finds itself back on its original trajectory?  In the rare case that this phenomenon occurs, it is characterized by two points: the point at which the particle first leaves the unperturbed trajectory, and the point at which the particle returns to it.  In other words, the two points are intersection points between a trajectory of the original dynamical system and a trajectory of the system described by the finite-time perturbation.  In the case where two such intersection points exist, we say that a leaf-wise intersection point exists.

In symplectic geometry, leaf-wise intersection points naturally occur between symplectic flows.  Moser himself worked in an exact symplectic framework, studying hypersurfaces of contact type in Liouville domains equipped with a dynamical system induced by choice of contact form.  He showed that, for a sufficiently small perturbation given by an exact symplectomorphism, such a hypersurface always contains a leaf-wise intersection point.

Moser's framework has been extended in many directions.  In \cite{albers-f}, Albers-Frauenfelder studied leaf-wise intersection points of hypersurfaces $\Sigma$ of restricted contact type in Liouville domains.  They showed that the number of leaf-wise intersection points of a sufficiently-small Hamiltonian occurring on $\Sigma$ is bounded below by the sum of the Betti numbers of $\Sigma$.  More generally, they showed that the existence of a leaf-wise intersection point of a fixed Hamiltonian, regardless of size, is guaranteed if an algebraic invariant called {\it Rabinowitz Floer homology} is non-vanishing on $\Sigma$.

This paper states the first results for leaf-wise intersection points on hypersurfaces of contact type within non-exact symplectic manifolds.  Following Albers-Frauenfelder, we probe for leaf-wise intersection points using various fixed-point Floer theories.  The only examples of such manifolds for which fixed-point Floer theory has been carefully examined are {\it negative line bundles} \cite{albers-kang}\cite{ritter-gromov}.  The first calculations are due to Ritter, who calculated another Floer invariant, called \emph{symplectic cohomology}, for these line bundles in \cite{ritter-gromov}.  It is for these manifolds, through the lens of symplectic cohomology, that we prove our first result.  

A negative line bundle $(E, \Omega)$ with negativity constant $k$ over a symplectic manifold $(M, \omega)$ is a complex line bundle 
\[
\mathfrak{p}:E\rightarrow M
\] 
whose Chern class satisfies
\[
c_1^E = -k[\omega]
\]
for some $k > 0$.  A choice of Hermitian metric picks out a family of contact hypersurfaces $\{\Sigma_r\}_{r>0}$, where $\Sigma_r$ is the circle subbundle appearing at radius $r$.  Recall that a symplectic manifold $(M, \omega)$ is monotone of monotonicity $c$ if
\[
c_1^{TM} = c[\omega]
\]
for some $c> 0$.
\begin{theorem}
\label{thm:linebundle}
Let $E \rightarrow M$ be a negative line bundle with negativity constant $k$ over a monotone symplectic manifold with monotonicity constant $c> k$.  If the symplectic cohomology of $E$ is non-zero, then for any compactly-supported Hamiltonian, the radius-$\frac{1}{\sqrt{k(c-k)\pi}}$ circle subbundle  contains a leaf-wise intersection point.
\end{theorem}

\begin{remark}
\label{rmk:chern}
Ritter showed in \cite{ritter-gromov} that the symplectic cohomology of $E$ is non-zero precisely when the action of $\mathfrak{p}^*c_1^E$ by quantum cup product on the quantum cohomology $QH^*(E)$ is not nilpotent.  In particular, this action is non-nilpotent when the base $M$ is a toric symplectic manifold \cite{ritter-fano}.
\end{remark}

\begin{corollary}
\label{cor:toric}
Let $E \rightarrow M$ be a negative line bundle with negativity constant $k$ over a monotone toric symplectic manifold with monotonicity constant $c> k$.  For any compactly-supported Hamiltonian, the radius-$\frac{1}{\sqrt{k(c-k)\pi}}$ circle subbundle  contains a leaf-wise intersection point.
\end{corollary}

\begin{example}
The line bundle $\mathcal{O}(-n)\rightarrow\mathbb{C}P^m$ satisfies the conditions of Corollary \ref{cor:toric} whenever $0 < n \leq m$.
\end{example}

The proof is inspired by the methods of Albers-Frauenfelder; it uses the non-vanishing of a Rabinowitz-Floer-esque invariant on $\Sigma_{\frac{1}{\sqrt{k(c-k)\pi}}}$.  This invariant vanishes on all other circle subbundles, and we conjecture that Theorem \ref{thm:linebundle} does not hold at any other radius: for any other radius $r$ there exist Hamiltonians with no leaf-wise intersection points on $\Sigma_r$.

The radius-$\frac{1}{\sqrt{k(c-k)\pi)}}$ circle bundle has appeared in other contexts as a ``special'' circle bundle.  In \cite{ritter-smith}, Ritter-Smith showed that, if the base $M$ is toric, the wrapped Fukaya category of $E$ is split-generated by a single Lagrangian torus inside $\Sigma_{\frac{1}{\sqrt{k(c-k)\pi}}}$, equipped with particular choices of local systems.  The presence of this Lagrangian implies that the Rabinowitz Floer invariant is non-zero \cite{venkatesh-rab}, which, as we will see in the proof of Theorem \ref{thm:linebundle}, in turn implies the existence of a leaf-wise intersection point.  Thus, Floer-essential Lagrangians detect leaf-wise intersection points.  A technical requirement for these arguments is that the Reeb flow of a circle subbundle is periodic of constant period, that is, that each circle subbundle is {\it Zoll}.

\begin{theorem}
\label{thm:lagrangian}
Let $E$ be a monotone symplectic manifold that is either compact or convex at infinity.  Let $\Sigma\subset E$ be a compact hypersurface of contact type with Zoll Reeb flow, containing a compact monotone Lagrangian $L$.  If there exists a local system $\gamma$ such that the Lagrangian Floer homology $HF^*(L, \gamma)$ is non-vanishing, then for any compactly-supported Hamiltonian, $\Sigma$ contains a leaf-wise intersection point.
\end{theorem}

\begin{example}
A silly example of this phenomenon comes from the Clifford torus $L$ in $\mathbb{C}P^1$.  This is a monotone, Floer essential Lagrangian, and so it cannot be displaced from itself by a Hamiltonian flow $\phi_H^1$ corresponding to a given Hamiltonian $H$.  $L$ is also a contact hypersurface, with periodic Reeb flow.  The intersection points of $\phi_H^1(L)$ with $L$ are precisely the leaf-wise intersection points of $H$.  
\end{example} 
A corollary to the proof of Theorem \ref{thm:lagrangian} removes the Zoll condition, at the expense of a precise statement.

\begin{corollary}
\label{cor:lagrangian}
Let $\Sigma\subset E$ be a compact hypersurface of contact type in a monotone symplectic manifold that is either compact or convex at infinity.  If $\Sigma$ contains a compact monotone Lagrangian $L$ such that $HF^*(L, \gamma)\neq 0$ for some local system $\gamma$, then for any compactly-supported Hamiltonian, there exists a hypersurface of contact type arbitrarily close to $\Sigma$ with a leaf-wise intersection point.
\end{corollary}
We restate Theorem \ref{thm:lagrangian} and Corollary \ref{cor:lagrangian} together in Section \ref{sec:lagrangian} as Theorem \ref{thm:lagrangians}.

The proofs of Theorems \ref{thm:linebundle} and \ref{thm:lagrangian} require monotonicity of the underlying manifold in different ways.  In Theorem \ref{thm:lagrangian}, monotonicity is required to define a closed-open map from the Rabinowitz Floer invariant to the {\it Lagrangian quantum cohomology of $L$}.  The latter has only been rigorously defined for monotone Lagrangians.  Using the machinery of \cite{fooo}, it should be possible to amend Theorem \ref{thm:lagrangian} to more general classes of symplectic manifolds and Floer-essential Lagrangians.  The proof of Theorem \ref{thm:linebundle} uses monotonicity as a technical assumption that allows us to precisely specify which hypersurface of contact-type contains a leaf-wise intersection point (see Lemma \ref{lem:actionest}).  

The constant $\frac{1}{\sqrt{k(c-k)\pi}}$ has an interesting alternative characterization.  As mentioned in Remark \ref{rmk:chern}, the  class $\mathfrak{p}^*c_1^E$ acts on the quantum cohomology $QH^*(E)$ via quantum cup product.  The largest eigenvalue $\lambda$ of this action has a valuation 
\[
ev(\lambda) = \frac{1}{{(c-k)}}.
\]  
Relaxing the monotonicity condition in Theorem \ref{thm:linebundle} yields an imprecise result in terms of this eigenvalue.

\begin{theorem}
\label{thm:weakplus}
Let $E$ be a weak\textsuperscript{+}-monotone negative line bundle.  Let $\lambda$ be a largest non-zero eigenvalue of the action of $\mathfrak{p}^*c_1^E$ by quantum cup product on $QH^*(E)$.  For any compactly-supported Hamiltonian, there exists a circle subbundle $\Sigma_r$ with $r \in \left(\sqrt{\frac{ev(\lambda)}{k\pi}}, \infty\right)$ such that $\Sigma_r$ has a leaf-wise intersection point. 
\end{theorem}

\subsection{Outline of paper}
In Section \ref{sec:floer} we recall the definitions of Hamiltonian Floer theory and symplectic cohomology and set notation.  We discuss capturing leaf-wise intersection points through Floer theory.  In Section \ref{sec:sketch} we introduce auxiliary Hamiltonians and sketch the proof of Theorem \ref{thm:linebundle}.  In Section \ref{sec:thm1} we prove Theorem \ref{thm:linebundle}.  In Section \ref{sec:thm2} we discuss a modified Floer theory and prove Theorem \ref{thm:weakplus}.  In Section \ref{sec:lagrangian} we introduce Lagrangian quantum cohomology and prove Theorem \ref{thm:lagrangian} and Corollary \ref{cor:lagrangian}, repackaged as Theorem \ref{thm:lagrangians}.

\subsection{Acknowledgements}
This work is supported by the National Science Foundation under Award No. 1902679.

\section{Floer theory and symplectic cohomology}
\label{sec:floer}

Let $\Sigma$ be a contact manifold with a fixed contact form $\alpha$.  The symplectization of $\Sigma$ is the manifold $(0, \infty)\times\Sigma$, equipped with the symplectic form $d(r\alpha)$, where $r$ is the coordinate on $(0, \infty)$.  Let $(E, \Omega)$ be a symplectic manifold, and assume that $E$ is either closed or there is a subset of the symplectization $[C, \infty)\times\Sigma$ such that $E$ is symplectomorphic to $[C, \infty)\times\Sigma$ outside of a compact set.  We define three Floer theories on $E$ over the {\it Novikov field over $\ZZ/2\ZZ$}, denoted by $\Lambda$ and defined by
\[
\Lambda = \left\{ \sum_{i=0}^{\infty}c_iT^{\alpha_i}\hspace{.2cm}\bigg|\hspace{.2cm} c_i\in\ZZ/2\ZZ, \RR\ni\alpha_i\rightarrow\infty\right\}.
\]
$T$ is a formal variable of degree equal to two times the minimum Chern number of $E$ on spheres:
\[
|T| = 2\cdot \min_{\substack{\beta\in\pi_2(E) \\ c_1^{TE}(\beta)\neq 0}}|c_1^{TE}(\beta)|.
\]
$\Lambda$ has a valuation $ev:\Lambda\rightarrow\RR\cup{\infty}$, given by
\[
ev\left( \sum_{i=0}^{\infty}c_iT^{\alpha_i}\right) = \min_{c_i\neq 0}\alpha_i.
\]
A Hamiltonian $H:E\times S^1\rightarrow\RR$ has an associated {\it Hamiltonian vector field} $X_H$, defined by
\[
dH(-) = \Omega(-, X_H).
\]
We denote the set of time-one orbits $x:S^1\rightarrow E$ of $X_H$ by $\cg{P}(H)$.  If the time-one orbits of $X_H$ satisfy a non-degeneracy condition, then an additional generic choice of $\Omega$-compatible almost-complex structure $J\in End(TE)$ produces a graded vector space generated by these orbits.  The graded vector space is called the {\it Floer complex} and is denoted by
\[
CF^*(H, J) := \bigoplus_{\substack{x\in\cg{P}(H)\\ \dot{x} = X_H(x)}}\Lambda\la x\ra.
\]
A choice of grading is determined by the Conley-Zehnder index $\mu_{CZ}$ on a lift of $\cg{L}E$ (see \cite{audin-d} for a definition of the Conley-Zehnder index).  Choose a capping disk $\tilde{x}$ for every $X_H$-periodic orbit $x$, and define $|x| = \mu_{CZ}(\tilde{x})$.  Then, for example,
\[
|T^{\alpha}x| = \alpha\cdot|T| + \mu_{CZ}(\tilde{x})
\]
gives an $\RR$-grading on $CF^*(H, J)$.  Up to chain isomorphism, the Floer complex is independent of choice of capping disk.

Let $x_+, x_-$ be two periodic orbits of $X_H$.  Fix a homotopy class $\beta\in\pi_2(E)$.  Let $\widehat{\cg{M}}_{\beta}(x_-, x_+)$ be the moduli space of maps $u:\RR\times S^1\rightarrow E$ satisfying Floer's equation
\begin{equation}
\label{eq:floer}
\frac{\dd u}{\dd s} + J\left(\frac{\dd u}{\dd t} - X_H\right) = 0,
\end{equation}
with asymptotes
\[
\lim_{s\rightarrow\pm\infty}u(s, t) = x_{\pm}(t),
\]
and such that the sphere formed by the connect sum $\tilde{x}_+\# u\#(-\tilde{x}_-)$ represents the homotopy class $\beta$:
\[
[\tilde{x}_+\# u\#(-\tilde{x}_-)] = \beta\in\pi_2(E).
\]
$\RR$ acts by translation on each $u$; let $\cg{M}_{\beta}(x_-, x_+)$ be the quotient of $\widehat{\cg{M}}_{\beta}(x_-, x_+)$ by this $\RR$ action.  $\cg{M}_{\beta}(x_+, x_-)$ is a zero-dimensional manifold when
\[
|x_-| - |x_+| + 2c_1^{TE}(\beta) = 1.
\]
If {\it Gromov compactness} is satisfied, $\cg{M}_{\beta}(x_+, x_-)$ is compact.  This happens under our assumptions if, for instance, $E$ is monotone.  The Floer differential is defined as
\[
\dd^{fl}(x_+) = \sum_{\substack{x_-\in\cg{P}(H) \\ \beta\in\pi_2(E) \\ |x_-| - |x_+| + 2c_1^{TE}(\beta)= 1}}\sum_{u\in \cg{M}_{\beta}(x_+, x_-)} T^{\beta}x_-
\]
and satisfies $(\dd^{fl})^2 = 0$.  Thus, $CF^*(H, J)$ is a cochain complex.

We now specialize to monotone negative line bundles.  Let $(E, \Omega)$ be a complex line bundle over a closed symplectic manifold $(B, \omega)$ with bundle map denoted by
\[
\mathfrak{p}: E\longrightarrow B.
\]
The line bundle $E$ is {\it negative} if there exists a constant $k > 0$ such that
\[
c_1^E = -k[\omega].
\]
Assume that $B$ is monotone, so that there exists a constant $c > 0$ (the monotonicity constant) such that
\[
c_1^{TB} = c[\omega].
\]
Oancea proved in \cite{oancea-leray} that there exists a Hermitian metric on $E$ with curvature form $\alpha$ such that
\[
d\alpha = \mathfrak{p}^*\omega.
\]
Let $r$ be the radial coordinate on the complex fibers with respect to this metric.  $E$ has a symplectic form $\Omega$ that, on the complement of the zero section, is written in coordinates as
\[
\Omega = d((1+k\pi r^2)\alpha).
\]
$\Omega$ extends smoothly over the zero section as the area form in the fiber direction and $\omega$ in the direction of the base.  It is straight-forward to check that this extension, also denoted by $\Omega$, is a symplectic form.  If $k < c$, $\Omega$ is monotone with monotonicity constant 
\[
\kappa := c-k.
\]

Fix a sequence $\{0 = R_0 < R_1 < R_2 < \dots\}$ converging to some fixed $R\in\RR_{>0}$.  Let $\{h_i:[0, \infty)\rightarrow\RR\}_{i\in\NN}$ be functions defined inductively by convexly smoothing the piecewise linear function
\[
h_i(r) = \left\{\begin{array}{cc} h_{i-1}(r) & r < R_i \\ (i + \frac{1}{2}) r + h_{i-1}(R_i) - (i + \frac{1}{2})R_i & r \geq R_i \end{array}\right..
\]
We define
\[
h_0(r) = \frac{1}{2}r.
\]
Assume without loss of generality that the sequence $\{R_i\}$ converges quickly enough that the family
$\{h_i\big|_{[0, R]}\}$ is universally bounded.  See Figure \ref{fig:unbddh}.
\begin{figure}[htpb!]
\centering
\begin{tikzpicture}[scale=5]
\draw (0,-.07) -- (0,-1.01916631) node [left = 1pt] {$\RR$};
\draw(0,-1.01916631)-- (1, -1.03660231)node [right = 1pt] {$r$};
\draw[violet] (0,-1.01916631)-- (.267948, -1);
\draw (.267948, -1.03660231)node [anchor=north]{{\tiny$R_1$}};
\draw (.641751, -1.03660231)node [anchor=north]{{\tiny$R_2$}};
\draw (.822168, -1.03660231)node [anchor=north]{{\tiny$R_3$}};
\draw[->, violet]  (.267948, -1) -- (1, -.9)node [right=1pt] {$h_0$};
\draw[->, violet] (.267948, -1) -- (1, -.57735)node [right=1pt] {$h_1$};
\draw[->, violet]  (.641751, -.784185) -- (1, -.37797)node [right=1pt] {$h_2$};
\draw[->, violet] (.822168, -.5796118) -- (1, -.2582)node [right=1pt] {$h_3$};
\draw[->, violet] (.911356, -.418414) -- (1, -.1796)node [right=1pt] {.};
\draw[->, violet] (.955751, -.2988102) -- (1, -.12599) node [right=1pt] {.};
\draw[->, violet] (.977884, -.212366) -- (1, -.08873) node [right=1pt] {.};
\draw (.267948, -1.02660231) -- (.267948, -1.04660231);
\draw (.641751, -1.02660231) -- (.641751, -1.04660231);
\draw (.822168, -1.02660231) -- (.822168, -1.04660231);
\end{tikzpicture}
\caption{The family of functions $\{h_i(r)\}$}
\label{fig:unbddh}
\end{figure}
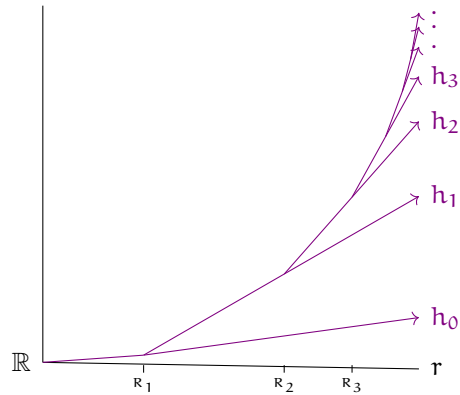

Let $G:M\rightarrow\RR$ be a Morse function.  On the complement of the zero section, $E$ is symplectomorphic to $(1, \infty)\times\Sigma$, where $\Sigma$ is the radius-one $S^1$-subbundle over $B$.  Let $(r, x)$ be coordinates on $(1, \infty)\times\Sigma$.  For each $i\in\NN$ define a Hamiltonian $H_i:(1, \infty)\times\Sigma\rightarrow\RR$ by
\[
H_i((r, x)) = h_i(k \pi r^2) + (1 + k\pi r^2)\mathfrak{p}^*G((r, x)).
\]
Extend $H_i$ to $E$ by setting 
\[
H_i = G
\]
on the zero-section.  For each $i\in\ZZ_{<0}$, similarly define $H_i:E\rightarrow\RR$ through
\[
H_i\big|_{(1, \infty)\times\Sigma} = -h_{-i-1}(k \pi r^2) + (1 +k\pi r^2)\mathfrak{p}^*G((r, x)).
\]
The periodic orbits of $H_i$ are not all non-degenerate; they are either
\begin{enumerate}
\item constant, non-degenerate orbits on the zero section, or
\item {\it transversally non-degenerate} $S^1$-orbits: orbits that come in $S^1$ families.
\end{enumerate}
The Floer differential $\dd^{fl}$ is modified by Morse-Bott methods to create a well-defined complex $CF^*(H_i)$, as in the methods defined for {\it autonomous Hamiltonians} in \cite{b-o}.  We refer there for details, or to \cite{venkatesh-thesis} for the study of autonomous Hamiltonians in the context of negative line bundles.

There are continuation maps $c_i:CF^*(H_i)\rightarrow CF^*(H_{i+1})$ for each $i\in\ZZ$ producing a directed system
\[
\dots\xrightarrow{c_{i-2}} CF^*(H_{i-1})\xrightarrow{c_{i-1}} CF^*(H_i) \xrightarrow{c_i}CF^*(H_{i+1})\xrightarrow{c_{i+1}}\dots.
\]
Continuation maps commute with the differential; thus, both the direct limit and the inverse limit of the directed system are chain complexes.  
The following useful Lemma is proved in \cite{venkatesh-thesis}.
\begin{lemma}
When $i \geq 0$, the continuation map $c_i$ can be chosen to be the canonical inclusion, taking an $S^1$-orbit $x:S^1\rightarrow E$ to itself.  When $i < -1$, the continuation map $c_i$ can be chosen to be the canonical projection.
\end{lemma}
Define the symplectic cohomology of the Hamiltonians $\{H_i\}$, denoted by $SH^*(H)$, as the homology of the direct limit
\[
SC^*(H) = \lim_{0\leq i\rightarrow\infty} CF^*(H_i).
\]
The condition $0\leq i$ is superfluous, but it will allow for cleaner formulations in proofs.  

The Floer complexes $CF^*(H_i)$ have a {\it filtration by action}.  Recall that the index of a periodic orbit $x$ of the Hamiltonian vector field $X_{H_i}$ is defined by fixing a capping disk $\tilde{x}$.  With this choice made, the action of $x$ is defined to be
\[
\cg{A}_{H_i}(x) = -\int_{D}\tilde{x}^*\Omega + \int_0^1H_i(x(t))dt.
\]
The action of a sum $X = \sum\limits_{j=0}^nC_jx_j \in CF^*(H_i)$ is given by
\begin{equation}
\label{eq:action}
\cg{A}_{H_i}(X) = \min_{C_j\neq 0} \left(ev(C_j) + \cg{A}(x_j)\right).
\end{equation}
Action is increased by the Floer differential, and so
\[
CF^*_{(a, \infty)}(H_i) := \left\{ X\in CF^*(H_i)\hspace{.2cm}\bigg|\hspace{.2cm} \cg{A}_{H_i}(X) > a\right\}
\]
is a subcomplex of $CF^*(H_i)$.  Continuation maps may be chosen to increase action; with these choices,
\[
SC_{(a, \infty)}^*(H) := \lim_{0\leq i\rightarrow\infty} CF^*_{(a, \infty)}(H_i)
\]
is a subcomplex of $SC^*(H)$.

Similarly, taking inverse limits, the complex
\[
SC^{(a, \infty)}_*(H) := \lim_{\substack{\leftarrow \\ i < 0}}CF^*_{(a, \infty)}(H_i)
\]
is a subcomplex of
\[
SC_*(H) := \lim_{\substack{\leftarrow \\ i < 0}}CF^*(H_i).
\]
If $a > a'$, there are inclusion maps
\[
SC^*_{(a, \infty)}(H)\hookrightarrow SC^*_{(a', \infty)}(H)
\]
and
\[
SC^{(a, \infty)}_*(H)\hookrightarrow SC^{(a', \infty)}_*(H).
\]
Thus, the action-filtered subcomplexes form directed systems.  We take the direct limit over all action windows $(a, \infty)$.  As direct limits commute with each other,
\[
\lim_{\substack{\rightarrow \\ a}}\lim_{0\leq i\rightarrow\infty} CF^*_{(a, \infty)}(H_i) = \lim_{0\leq i\rightarrow\infty} \lim_{\substack{\rightarrow \\ a}}CF^*_{(a, \infty)}(H_i) \simeq SC^*(H). 
\]
However,
\begin{equation}
\label{eq:completedcomplex}
\widehat{SC_*}(H) := \lim_{\substack{\rightarrow \\ a}}\lim_{\substack{\leftarrow \\ i < 0}}CF^*_{(a, \infty)}(H_i)
\end{equation}
is not necessarily equal to $SC_*(H)$.  Denote the homology of (\ref{eq:completedcomplex}) by $\widehat{SH^*}(H)$, and call it {\it completed symplectic homology}.

The following computation was performed by Ritter in \cite{ritter-gromov}.
\begin{theorem}[Ritter \cite{ritter-fano}]
\label{thm:ritter}
The symplectic cohomology of a monotone negative line bundle over a toric base is non-vanishing:
\[
SH^*(H)\neq 0.
\]
\end{theorem}
Ritter's method of proof implies the dual theorem:
\[
SH_*(H)\neq 0.
\]  
However, these non-vanishing theorems do not always hold for completed symplectic homology.
\begin{theorem}
\label{thm:vanishing}
The completed symplectic homology of a monotone negative line bundle vanishes for small enough radii:
\[
\widehat{SH_*}(H) = 0
\]
whenever $R < \frac{1}{\sqrt{k\kappa \pi}}$.
\end{theorem}
\begin{remark}
Theorem \ref{thm:vanishing} is reminiscent of Theorem 1.2 in \cite{albers-kang}, which proved a non-vanishing result for a related homology theory termed {\it Rabinowitz Floer homology}.  Theorem \ref{thm:vanishing} is also the natural extension of Theorem 1 in \cite{venkatesh-neg}.  The proof employs, with very little modification, the methods developed in \cite{albers-kang}.
\end{remark}

\subsection{Hamiltonians with truncated support}
\label{subsec:trunc}

In preparation for setting up Floer theory for leaf-wise intersection points, consider the following modification to the Floer chain complex.  Let $\rho:\RR\rightarrow\RR$ be a bump function with support in $[0, 1/2]$ that integrates to one:
\[
\int_{\RR}\rho(t)dt = 1.
\]
Define Hamiltonians 
\[
\cg{H}_n = \rho(t)H_n.
\]
The periodic orbits of $\cg{H}_n$ are precisely those of $H_n$.  Standard results in Floer theory, which we relegate to the Appendix, show that the symplectic chain complex $SC^*(\cg{H})$ is well-defined, and the non-vanishing of $SH^*(H)$ in Theorem \ref{thm:ritter} implies the non-vanishing
\begin{equation}
\label{eq:truncatediso}
SH^*(\cg{H})\neq 0
\end{equation}
where $SH^*(\cg{H})$ is the symplectic cohomology defined using the Hamiltonians $\{\cg{H}_n\}$.

Analogously to Theorem \ref{thm:vanishing},
\begin{theorem}
\label{thm:truncated-vanishing}
The completed symplectic homology of $\cg{H}$ vanishes whenever $ R < \frac{1}{\sqrt{k\kappa\pi}}:$
\[
\widehat{SH_*}(\cg{H}) = 0.
\]
\end{theorem}

\subsection{Leaf-wise intersections in Floer theory}

Leaf-wise intersection points were first characterized through Floer theory by Albers-Frauenfelder in \cite{albers-f}.  Kang subsequently adjusted their framework to study leaf-wise intersection points through symplectic cohomology \cite{kang}.  In this section, we recall their methods and prove some necessary, albeit technical, Lemmas on perturbing away degenerate orbits.

Let $F:E\times \RR\rightarrow\RR$ be a compactly supported Hamiltonian.  Let $\Sigma_R = \{R\}\times\Sigma$ be the circle subbundle at radius $R$.  A {\it leafwise intersection point of $F$ on $\Sigma_R$} is given by a flow line of the Hamiltonian flow corresponding to $F$
\[
x:[0, 1]\rightarrow E
\] 
where $x(0)$ and $x(1)$ lie on the same Reeb trajectory in $\Sigma_R$.  The Reeb trajectories transcribe the fibers of the subbundle $\Sigma_R$; $x(1)$ is a leaf-wise intersection point if and only if $x(0)$ and $x(1)$ lie on the same circle fiber.  We are interested in leaf-wise intersection points on the fixed contact hypersurface $\Sigma_{\frac{1}{\sqrt{k\kappa\pi}}}$.

Assume without loss of generality that the support of $F$ lies in $E\times[1/2, 1]$.  As in Subsection \ref{subsec:trunc}, let $\rho:\RR\rightarrow\RR$ be a bump function with support in $[0, 1/2]$ that integrates to one:
\[
\int_{\RR}\rho(t)dt = 1.
\]
Define Hamiltonians $\left\{H_i^F:E\times [0, 1]\rightarrow\RR\right\}_{i\in\ZZ}$ by
\[
H_i^F = \rho(t)h_i(k\pi r^2) + F.
\]
The dynamics of $H_i^F$ are governed by $\rho(t)h_i(k \pi r^2)$ when $t\in[0, 1/2]$ and by $F$ when $t\in[1/2, 1]$.  Denoting by $R_{\alpha}$ the Reeb vector field of $\alpha$, the Hamiltonian vector field of $H_i^F$ is
\[
\left(X_{H_i^F}\right)_t = \left\{\begin{array}{cc} \rho(t)h'_i(k \pi r^2) R_{\alpha} & t\in[0, 1/2] \\ \left(X_F\right)_t & t\in[1/2, 1] \end{array}\right..
\]
The periodic orbits of $H_i^F$ that begin on a hypersurface $\Sigma_R$ correspond to leaf-wise intersection points on $\Sigma_R$ (see Figure \ref{fig:lw}).  Denoting the Reeb flow of $(\Sigma_R, \alpha)$ by $\phi_R^t$, these periodic orbits are in bijection precisely with the leaf-wise intersection points $x:[0, 1]\rightarrow E$ satisfying 
\[
x(1) = \phi_1^{h'_i(k\pi r)}(x(0)).
\]
Note that the periodic orbits of $H_i^F$ that begin on the zero section are periodic orbits of $F$.  

\begin{figure}
\centering
\includegraphics[scale=.5]{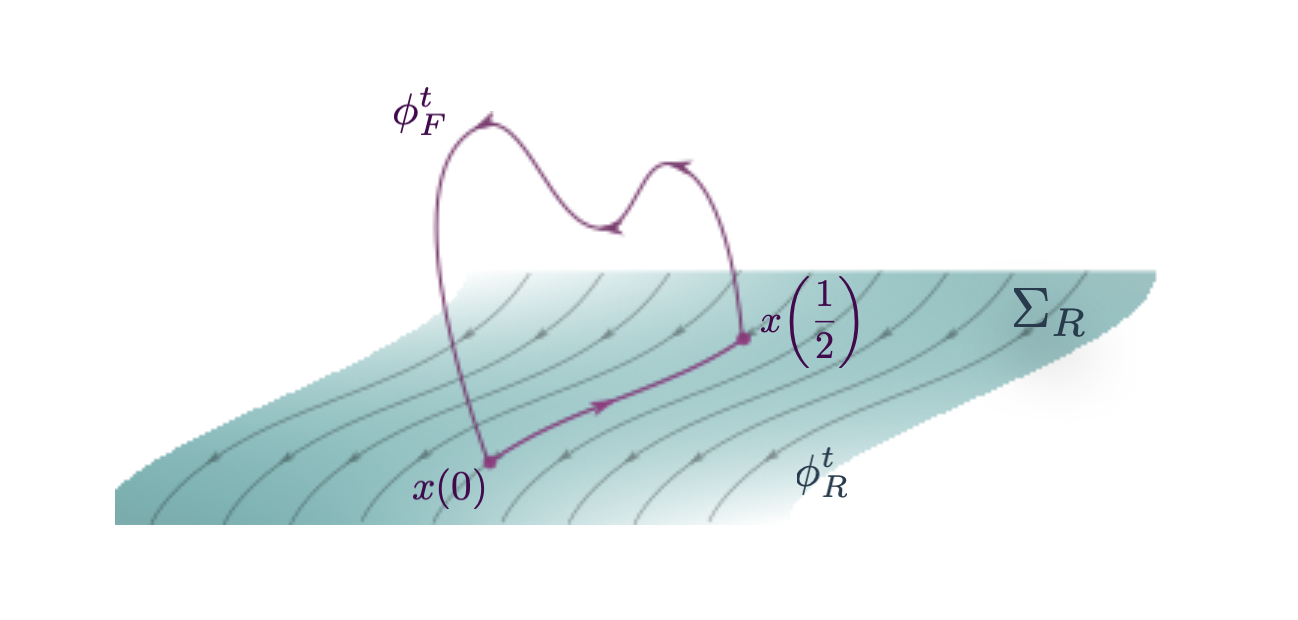}
\caption{A leaf-wise intersection point exhibited as a periodic orbit $x(t)$ of $X_{H_i^F}$.  The purple orbit $x(t)$ travels backwards along a Reeb trajectory, pauses at $x(1/2)$, and then travels according to the flow $\phi_F^t$ of the Hamiltonian vector field  $(X_F)_t$.}
\label{fig:lw}
\end{figure}

We want to show that $F$ has a leaf-wise intersection point on $\Sigma_{\frac{1}{\sqrt{k\kappa\pi}}}$.  Abusing terminology, call such points the {\it leaf-wise intersection points of $F$}, dropping mention of the specific fixed hypersurface.  If the periodic orbits of $H_i^F$ beginning on $\Sigma_R$ are degenerate for any $i$, it follows that these periodic orbits exist, and we conclude that $F$ has a leaf-wise intersection point.  So, for the rest of the paper, we will assume that all such periodic orbits are non-degenerate.

To set up Hamiltonian Floer theory for $H_i^F$, all periodic orbits of $H_i^F$ must be non-degenerate, not just those beginning on $\Sigma_R$.  While this does not hold {\it a priori} for $F$, it holds for a suitable perturbation of $F$.  The existence of such a suitable perturbation follows from a standard result in Floer theory.

\begin{lemma}
\label{lem:perturbopen}
Fix a Hamiltonian $H$.  Let $U\subset E\times \RR$ be an open set.  Let $\cg{H}(U)$ be the set of Hamiltonians with support in $U$.  The subset of Hamiltonians $K\in\cg{H}(U)$ for which every period orbit of $x(t)$ of $K+H$ is non-degenerate if $(x(t), t)$ passes through $U$ is comeager in $\cg{H}(U)$.
\end{lemma}

\begin{proof}
See the Appendix in \cite{albers-b-w}.
\end{proof}

\begin{lemma} 
\label{lem:perturbF}
There exists a perturbation $\tilde{F}$ of $F$ such that 
\begin{enumerate}
\item the support of $\tilde{F}$ lies in $E\times[1/2, 1]$,
\item the leaf-wise intersection points are unchanged,
\item the periodic orbits of $\tilde{F}+\rho(t)(n-\frac{1}{2})k\pi r^2$ are non-degenerate for any $n\in\ZZ$, and 
\item the periodic orbits of $\tilde{F}$ contained in $\text{supp}(\tilde{F})$ are non-degenerate.
\end{enumerate}
\end{lemma}

\begin{proof}
Let $\phi_t^n$ be the time-$t$ flow of the Hamiltonian vector field of $F+\rho(t)(n-\frac{1}{2})k\pi r^2$, and let $B$ be the subset
\[
B = \left\{(\phi_t^n(x), t) \in E\times [0, 1] \hspace{.2cm}\bigg|\hspace{.2cm} x\in\Sigma_{\frac{1}{\sqrt{k\kappa\pi}}} \right\}\cup\bigg( E\times[0, 1/2]\bigg).
\]
Note that $B$ does not depend on $n$.

Let $U_B$ be the open subset
\[
U_B = \big(E\times [0, 1]\big)\setminus B.
\]  
By Lemma \ref{lem:perturbopen}, there exists a comeager set of perturbations $\cg{M}(n)$ of $F+\rho(t)(n-\frac{1}{2})k\pi r^2$ such that every periodic orbit passing through $U_B$ is non-degenerate.  By assumption, each perturbation in $\cg{M}(n)$ occurs for $t\in[1/2, 1]$, the interval on which $F = F+\rho(t)(n-\frac{1}{2})k\pi r^2$.  We therefore view elements of $\cg{M}(n)$ as direct perturbations of the Hamiltonian $F$.  Define the intersection
\[
\cg{M}(U_B) = \bigcap_{n\in\ZZ}\cg{M}(n).
\]  
As a countable intersection of comeager sets, $\cg{M}(U_B)$ is also comeager.  Let
$
\cg{M}(F) 
$
be a comeager set of perturbations of $F$ with non-degenerate periodic orbits.  The periodic orbits of $F$ contained in $B\bigcap\text{supp}(F)$ are leaf-wise intersection points, and so are non-degenerate by assumption.  There therefore exists a comeager subset of $\cg{M}(F)$ comprised of perturbations of $F$ with support in $U_B$.  It follows that $\cg{M}(U_B)\cap \cg{M}(F)$ is itself comeager.  Choose a perturbation \[\tilde{F}\in\cg{M}(U_B)\cap\cg{M}(F).\]  By construction, the support of $\tilde{F}$ is in $E\times[1/2, 1]$.

Let
\[
B' = \left\{(\phi_t^n(x), t) \in E\times [0, 1] \hspace{.2cm}\big|\hspace{.2cm} x\in\Sigma_{\frac{1}{\sqrt{k\kappa\pi}}} \right\}.
\]  Define $\cg{M}(U_{B'})$ analogously to $\cg{M}(U_B)$.  Note that a periodic orbit $x(t)$ of any Hamiltonian of the form $F + \rho(t)(n-\frac{1}{2})k\pi r^2$ that passes through $U_{B'}$ also passes through $U_{B}$.  Therefore, $\tilde{F}$ is an element of the larger set $\cg{M}(U_{B'})$.

By construction the perturbation $\tilde{F}$ satisfies the conditions of the Lemma.

\end{proof}

\begin{lemma} 
\label{lem:nondeg}
Define Hamiltonians $H_n^{\tilde{F}} = \tilde{F} + \rho(t)h_n(k\pi r^2)$.  Recall the data of radii $\{R_1, R_2, \dots\}$ in the definition of the functions $\{h_n\}$.  The periodic orbits of $H_n^{\tilde{F}}$ with starting point within
\[
E\setminus\bigg([R_1, R_n]\times\Sigma_{\frac{1}{\sqrt{k\kappa\pi}}}\bigg).
\]
are non-degenerate.
\end{lemma}

\begin{proof}
By construction, the function $h_n(k \pi r^2)$ is linear in $k \pi r^2$ on the cylindrical end modeled as $[R_n, \infty)\times\Sigma$ and on the radius-$R_1$ disk bundle.  The dynamics of $\tilde{F} + \rho(t)h_n(k\pi r^2)$ coincide with the dynamics of $\tilde{F} + \rho(t)(n- \frac{1}{2})k\pi r^2$ on any trajectory beginning in $[R_n, \infty)\times\Sigma$, and they coincide with the dynamics of $\tilde{F} + \rho(t)\frac{1}{2}k\pi r^2$ within the disk bundle of radius $R_1$.  By Lemma \ref{lem:perturbF}, the periodic orbits in these sets are non-degenerate.

\end{proof}

\section{Auxiliary Hamiltonians and a proof sketch}
\label{sec:sketch}

\addtocounter{theorem}{-6}
Recall Theorem \ref{thm:linebundle}.
\begin{theorem}
Let $E \rightarrow M$ be a negative line bundle with negativity constant $k$ over a monotone symplectic manifold with monotonicity constant $c> k$.  If the symplectic cohomology of $E$ is non-zero, then for any compactly-supported Hamiltonian, the radius-$\frac{1}{\sqrt{k\kappa\pi}}$ circle subbundle  contains a leaf-wise intersection point.
\end{theorem}
\addtocounter{theorem}{+5}

The proof of Theorem \ref{thm:linebundle} proceeds in two steps:
\begin{enumerate}
\item show that, for any annulus subbundle $A\subset E$ containing $\Sigma_{\frac{1}{\sqrt{k\kappa\pi}}}$, there exists some circle subbundle $\Sigma\subset A$ containing a leaf-wise intersection point; and
\item show that there exists a sequence of leaf-wise intersection points whose limit exists and is a leaf-wise intersection point on $\Sigma_{\frac{1}{\sqrt{k\kappa\pi}}}$.
\end{enumerate}

The second step is straight-forward, and is given by Proposition \ref{prop:nested}.
The first step proceeds by a contradiction, arising through the subtle maneuverings of different families of Hamiltonians.  In this section, we collect the families of Hamiltonians, and we sketch the proof of Step 1.

\subsection{The Hamiltonians}
\subsubsection{$H_n$ and $K_n$}
Fix a radius $P$ and a radius $R$, with $0 < P < \frac{1}{\sqrt{k\kappa\pi}} < R$.  Fix monotone-increasing sequences of real numbers $P_1, P_2, \dots$ and $R_1, R_2, \dots$ with
\[
0 < P_1 < P_2 < \dots \rightarrow P < R_1 < R_2 < \dots\rightarrow R
\]
and converging to $P$, respectively $R$.

For $n \geq 0$, define $H_n$ and $H_n^F$ using the functions described in Section \ref{sec:floer}.  We set
\[
H_n = \rho(t)\left(h_n(k\pi r^2) + \mathfrak{p}^*G\right)
\]
and
\[
H_n^F = \rho(t)h_n(k\pi r^2) + F.
\]
Thus, $H_n^F$ is precisely as defined in Section \ref{sec:floer}, and, in the notation of Subsection \ref{subsec:trunc}, $H_n$ is now the Hamiltonian $\cg{H}_n$.  See Figure \ref{fig:unbddh} to recall the definition of $h_n$.  

For $n\geq0$, define $K_n$ and $K_n^F$ analogously to the Hamiltonians $H_n$ and $H_n^F$, except 
\begin{itemize}
\item the radial functions $k_n$ are negated and
\item are defined through the radii sequence $P_1, P_2, \dots$ in lieu of $R_1, R_2, \dots$.
\end{itemize}  
Precisely, there are functions $k_n:\RR\rightarrow\RR$, as depicted in Figure \ref{fig:unbddk}, and a Morse function $G$ on the base of the bundle $B$ such that
\[
K_n = \rho(t)\left(k_n(k\pi r^2) + (1 + k\pi r^2)\mathfrak{p}^*G\right)
\]
and
\[
K_n^F = \rho(t)k_n(k\pi r^2) + F.
\]
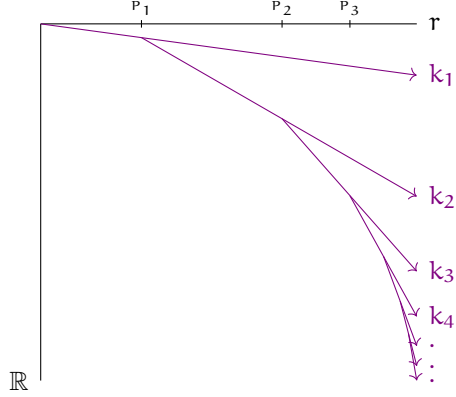
\begin{figure}[htpb!]
\centering
\begin{tikzpicture}[scale=5]
\draw (0,-.07) -- (0,-1.01916631) node [left = 1pt] {$\RR$};
\draw (0, -.07) -- (1, -.07)node [right = 1pt] {$r$};
\draw[->, violet] (0, -.07) -- (1, -0.20660321)node [right=1pt] {$k_1$};
\draw[->, violet] (.267948, -0.10660321) -- (1, -0.52925231)node [right=1pt] {$k_2$};
\draw[->, violet]  (.641751, -0.32241731) -- (1, -0.72863231)node [right=1pt] {$k_3$};
\draw[->, violet] (.822168, -0.5274841) -- (1, -0.8488959)node [right=1pt] {$k_4$};
\draw[->, violet] (.911356, -0.68818831) -- (1,-0.92700231)node [right=1pt] {.};
\draw[->, violet] (.955751, -0.80779211) -- (1, -0.98061231) node [right=1pt] {.};
\draw[->, violet] (.977884, -0.89423631) -- (1, -1.01916631) node [right=1pt] {.};
\draw (.267948, -.07)node [anchor=south]{{\tiny$P_1$}};
\draw (.641751, -.07)node [anchor=south]{{\tiny$P_2$}};
\draw (.822168, -.07)node [anchor=south]{{\tiny$P_3$}};
\draw (.267948, -.08) -- (.267948, -.06);
\draw (.641751, -.08) -- (.641751, -.06);
\draw (.822168, -.08) -- (.822168, -.06);
\end{tikzpicture}
\caption{The family of functions $\{k_n(r)\}$}
\label{fig:unbddk}
\end{figure}
Fix continuation maps
\begin{align*}
CF^*(H_n)\rightarrow CF^*(H_{n+1}), &\hspace{1cm} CF^*(H_n^F)\rightarrow CF^*(H_{n+1}^F),
\\
CF^*(K_m)\rightarrow CF^*(K_{m-1}), &\hspace{1cm}CF^*(K_m^F)\rightarrow CF^*(K_{m-1}^F).
\end{align*}

\subsubsection{$\check{H}_{m, n}$ and $\check{K}_{m, n}$}
Fix one more radius $Q$, with $P < Q < R$.
Let $\check{k}_{m, n}$ be a function that is
\begin{enumerate}
\item equal to $k_m$ on $[0, P_m]$,
\item convex on $[P, \infty)$,
\item and a translation of $h_n$ on $[Q, \infty)$.
\end{enumerate}
See Figure \ref{fig:checksk}.  Define functions $\check{K}_{m, n}$ and $\check{K}_{m, n}^F$ by
\[
\check{K}_{m, n} = \rho(t)\left(\check{k}_{m, n}(k\pi r^2) + (1+k\pi r^2)\mathfrak{p}^*G\right)
\]
and
\[
\check{K}_{m, n}^F = \rho(t)\check{k}_{m, n}(k\pi r^2) + F.
\]
Let
\[
\check{h}_{m, n} = \check{k}_{m, n} + \min_{r\geq0}\check{k}_{m, n}(r),
\]
as in Figure \ref{fig:checksh}.  Define functions $\check{H}_{m, n}$ and $\check{H}_{m, n}^F$ by
\[
\check{H}_{m, n} = \rho(t)\left(\check{h}_{m, n}(k\pi r^2) + (1+k\pi r^2)\mathfrak{p}^*G\right)
\]
and
\[
\check{H}_{m, n}^F = \rho(t)\check{h}_{m, n}(k\pi r^2) + F.
\]
These are variations on the ``v-shaped'' Hamiltonians appearing in \cite{c-f-o}.
\begin{figure}[htpb!]
\centering
\subfigure[$\check{k}_{1, 4}$]
{
\label{fig:checksk}
\begin{tikzpicture}[scale=5]
\draw (0,-.07) -- (0,-1.01916631) node [left = 1pt] {$\RR$};
\draw (0, -.07) -- (1, -.07)node [right = 1pt] {$r$};
\draw[-, violet] (0, -.07) -- (0.133974, -0.088301605)node [right=1pt] {};
\draw[-, violet] (0.133974, -0.088301605)-- (0.3208755, -0.196208655)node [right=1pt] {};
\draw[-, violet] (0.3208755, -0.196208655) -- (0.411084, -0.29874205)node [right=1pt] {};
\draw[-, violet]  (0.411084, -0.29874205) -- (0.683517, -0.8136412325)node [right=1pt] {};
\draw[-, violet]  (0.683517, -0.8136412325) -- (0.817491,-0.7953400736 ) node [right=1pt] {};
\draw[-, violet] (0.817491,-0.7953400736 ) -- ( 1, -0.6815038011) node [right=1pt] {};
\draw (0.133974, -.07)node [anchor=south]{{\tiny$P_1$}};
\draw (0.3208755, -.07)node [anchor=south]{{\tiny$P_2$}};
\draw (0.411084, -.07)node [anchor=south]{{\tiny$P_3$}};
\draw (0.817491, -.07)node [anchor=south]{{\tiny$R_1$}};
\draw (0.683517, -.07)node [anchor=south]{{\tiny$Q_{}$}};
\draw (0.54, -.07)node [anchor=south]{{\tiny$P_{}$}};
\draw[-] (0.133974, -.08) -- (0.133974, -.06);
\draw[-] (0.3208755, -.08) -- (0.3208755, -.06);
\draw[-] (0.411084, -.08) -- (0.411084, -.06);
\draw[-] (0.817491, -.08) -- (0.817491, -.06);
\draw[-] (0.683517, -.08) -- (0.683517, -.06);
\draw[-] (0.54, -.08) -- (0.54, -.06);
\end{tikzpicture}

}
\hspace{2cm}
\subfigure[$\check{h}_{1, 4}$]
{
\label{fig:checksh}
\begin{tikzpicture}[scale=5]
\draw (0,-.07) -- (0,-1.01916631) node [left = 1pt] {$\RR$};
\draw (0,-0.8136412325) -- (1,-0.8136412325)node [right = 1pt] {$r$};
\draw[-, violet] (0, -.07) -- (0.133974, -0.088301605)node [right=1pt] {};
\draw[-, violet] (0.133974, -0.088301605)-- (0.3208755, -0.196208655)node [right=1pt] {};
\draw[-, violet] (0.3208755, -0.196208655) -- (0.411084, -0.29874205)node [right=1pt] {};
\draw[-, violet]  (0.411084, -0.29874205) -- (0.683517, -0.8136412325)node [right=1pt] {};
\draw[-, violet]  (0.683517, -0.8136412325) -- (0.817491,-0.7953400736 ) node [right=1pt] {};
\draw[-, violet] (0.817491,-0.7953400736 ) -- ( 1, -0.6815038011) node [right=1pt] {};
\draw (0.133974, -0.8136412325)node [anchor=north]{{\tiny$P_1$}};
\draw (0.3208755, -0.8136412325)node [anchor=north]{{\tiny$P_2$}};
\draw (0.411084, -0.8136412325)node [anchor=north]{{\tiny$P_3$}};
\draw (0.817491, -0.8136412325)node [anchor=north]{{\tiny$R_1$}};
\draw (0.683517,-0.8136412325)node [anchor=north]{{\tiny$Q$}};
\draw (0.54, -0.8136412325)node [anchor=north]{{\tiny$P_{}$}};
\draw[-] (0.133974, -.823) -- (0.133974, -.803);
\draw[-] (0.3208755, -.823) -- (0.3208755, -.803);
\draw[-] (0.411084, -.823) -- (0.411084, -.803);
\draw[-] (0.817491, -.823) -- (0.817491, -.803);
\draw[-] (0.683517, -.823) -- (0.683517, -.803);
\draw[-] (0.54, -.823) -- (0.54, -.803);
\end{tikzpicture}
}
\caption{``v-shaped'' Hamiltonians}
\end{figure}
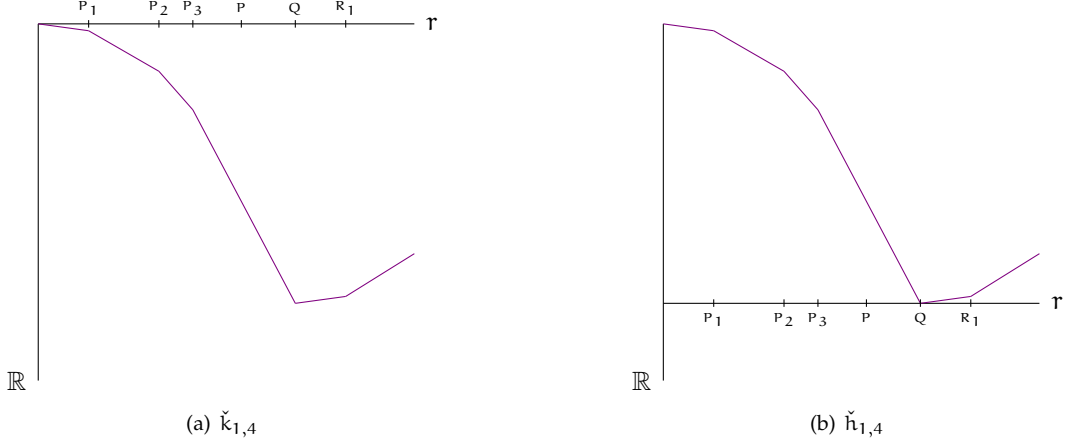
For each $m\in\ZZ_{<0}$ and $n\in\NN$, define a generic homotopy $\check{K}^s_{m, n}$ between $\check{K}_{m, n}$ and $\check{K}_{m-1, n}$ that is
\begin{enumerate}
    \item monotone increasing everywhere and
    \item is equal to $\check{K}_{m, n} + k(s)$ for an $s$-dependent function $k$ on $E\setminus\left([P_{m-1}, \infty)\times\Sigma\right)$ and on $[Q, \infty)\times\Sigma$.
\end{enumerate}
Let
\[
c_m:CF^*(\check{K}_{m, n})\rightarrow CF^*(\check{K}_{m-1, n})
\]
be the continuation map induced by $\check{K}^s_{m, n}$.  We suppress the $n$-index from the notation.

Similarly define continuation maps
\[
c_n:CF^*(\check{K}_{m, n})\rightarrow CF^*(\check{K}_{m, n+1})
\]
The continuation maps $c_n$ and $c_m$ induce a barrage of continuation maps
\begin{align*}
CF^*(\check{H}_{m, n})\rightarrow CF^*(\check{K}_{m-1, n}^F), &\hspace{1cm}CF^*(\check{K}_{m, n}^F)\rightarrow CF^*(\check{K}_{m, n+1}^F)
\\
CF^*(\check{H}_{m, n})\rightarrow CF^*(\check{H}_{m-1, n}), &\hspace{1cm}CF^*(\check{H}_{m, n})\rightarrow CF^*(\check{H}_{m, n+1})
\\
CF^*(\check{H}_{m, n}^F)\rightarrow CF^*(\check{H}_{m-1, n}^F), &\hspace{1cm}CF^*(\check{H}_{m, n}^F)\rightarrow CF^*(\check{H}_{m, n+1}^F).
\end{align*}

\subsection{Sketch of proof}

Assume for contradiction that there does not exist a circle subbundle contained inside the annulus subbundle $[P_1, R]\times\Sigma$ with a leaf-wise intersection point of $F$.  By Lemma \ref{lem:nondeg}, we may assume that the periodic orbits of $H_n^F, K_m^F, \check{H}_{m, n}^F,$ and $\check{K}_{m, n}^F$ are all non-degenerate for every $n$ and $m$.

Denote the disk bundle of radius $R$ by $D_R$.  Let
\[
C = \sup_{x\in D_R}|F(x)| + \sup_{x\in D_R}(1+k\pi R^2)|\mathfrak{p}^*G(x)|.
\]
We show in Subection \ref{subsec:chain} the following chain of quasi-isomorphisms and quasi-inclusions
\begin{align*}
\lim_{\substack{\leftarrow\\ m}}CF^*_{(a, \infty)}(K_m) &\overset{(1)}{\simeq} \lim_{\substack{\rightarrow \\ n}}\lim_{\substack{\leftarrow\\ m}}CF^*_{(a, \infty)}(\check{K}_{m, n})\overset{(2)}{\hookrightarrow}  \lim_{\substack{\rightarrow \\ n}}\lim_{\substack{\leftarrow\\ m}} CF^*_{(a + C, \infty)}(\check{K}_{m, n}^F) \\ &\overset{(3)}{\simeq} \lim_{\substack{\rightarrow \\ n}}\lim_{\substack{\leftarrow\\ m}} CF^*_{(a + C, \infty)}(\check{H}_{m, n}^F) \overset{(4)}{\hookrightarrow}\lim_{\substack{\rightarrow \\ n}}\lim_{\substack{\leftarrow\\ m}} CF^*_{(a + 2C, \infty)}(\check{H}_{m, n}) \overset{(5)}{\simeq} \lim_{\substack{\rightarrow \\ n}}CF^*_{(a + 2C, \infty)}(H_n).
\end{align*}
Taking the direct limit over action window, the quasi-inclusions become quasi-isomorphisms, so that there is a quasi-isomorphism
\[
\lim_{\substack{\rightarrow \\ a}}\lim_{\substack{\leftarrow\\ m}}CF^*_{(a, \infty)}(K_m)\simeq \lim_{\substack{\rightarrow \\ a}}\lim_{\substack{\rightarrow \\ n}}CF^*_{(a + 2C, \infty)}(H_n).
\]  
Taking homology yields an isomorphism
\[
\widehat{SH_*}(K) := H\left(\lim_{\substack{\rightarrow \\ a}}\lim_{\substack{\leftarrow\\ m}}CF^*_{(a, \infty)}(K_m)\right) \simeq H\left(\lim_{\substack{\rightarrow \\ a}}\lim_{\substack{\rightarrow \\ n}}CF^*_{(a + 2C, \infty)}(H_n)\right) =: SH^*(H).
\]
We show in Proposition \ref{prop:shvanishes} that $\widehat{SH_*}(K) = 0$.  Thus, if the symplectic cohomology $SH^*(H)$ is non-zero, we reach a contradiction.  We conclude that there exists a leaf-wise intersection point somewhere within the annulus $[P_1, R]\times\Sigma$.

Choose a sequence of nested annuli
\[
[P_1, R]\times\Sigma \supset [P_1^1, R^1]\times\Sigma \supset [P_1^2, R^2] \times\Sigma\supset \dots 
\]
whose intersection is 
\[
 \Sigma_{\frac{1}{\sqrt{k\kappa\pi}}} = \bigcap\limits_{i=1}^{\infty} [P_1^i, R^i]\times\Sigma.
\]
Let $\{x_i(t)\}$ be a sequence of leaf-wise intersection points, such that $x_i(0) \in [P_1^i, R^i]\times\Sigma$.  We show in Proposition \ref{prop:nested} that the sequence $\{x_i\}$ has a subsequence that converges to a leaf-wise intersection point on $\Sigma_{\frac{1}{\sqrt{k\kappa\pi}}}$.

\section{Proof of Theorem \ref{thm:linebundle}}
\label{sec:thm1}

\subsection{A chain of isomorphisms}
\label{subsec:chain}
In this subsection we fix an action window $(a, \infty)$ and prove the following chain of isomorphisms and inclusions.
\begin{align*}
\lim_{\substack{\leftarrow\\ m}}HF^*_{(a, \infty)}(K_m) &\overset{(1)}{\simeq} \lim_{\substack{\rightarrow \\ n}}\lim_{\substack{\leftarrow\\ m}}HF^*_{(a, \infty)}(\check{K}_{m, n})\overset{(2)}{\hookrightarrow}  \lim_{\substack{\rightarrow \\ n}}\lim_{\substack{\leftarrow\\ m}} HF^*_{(a + C, \infty)}(\check{K}_{m, n}^F) \\ &\overset{(3)}{\simeq} \lim_{\substack{\rightarrow \\ n}}\lim_{\substack{\leftarrow\\ m}} HF^*_{(a + C, \infty)}(\check{H}_{m, n}^F) \overset{(4)}{\simeq}\lim_{\substack{\rightarrow \\ n}} HF^*_{(a + C, \infty)}(H_n^F) \overset{(5)}{\hookrightarrow} \lim_{\substack{\rightarrow \\ n}}HF^*_{(a + 2C, \infty)}(H_n).
\end{align*}

\subsubsection{The isomorphism (1)}

We prove the following Proposition, illustrated in Figure \ref{fig:qi(1)}.
\begin{proposition}
\label{prop:qi(1)}
There is a quasi-isomorphism
\[
\lim_{\substack{\leftarrow \\ m}}CF^*_{(a, \infty)}(K_m) \cong \lim_{\substack{\rightarrow \\ n}}\lim_{\substack{\leftarrow \\ m}} CF^*_{(a, \infty)}(\check{K}_{m,n})
\]
for any $a\in\RR$.
\end{proposition}
The proof of Proposition \ref{prop:qi(1)} proceeds through the study of an intermediary subcomplex, $CF^*_P(\check{K}_{m, n})$.  Define $CF^*_P(\check{K}_{m, n})$ as the vector space generated by the subset of orbits of $CF^*_P(\check{K}_{m, n}$ contained in $E\setminus[P, \infty)\times\Sigma$.  The Floer differential and continuation map on $CF^*_P(\check{K}_{m, n})$ are characterized by Lemma \ref{lem:remain} and the immediate Corollary \ref{cor:subcplx}.

\begin{figure}[!htbp]
\centering
\includegraphics[scale=.7]{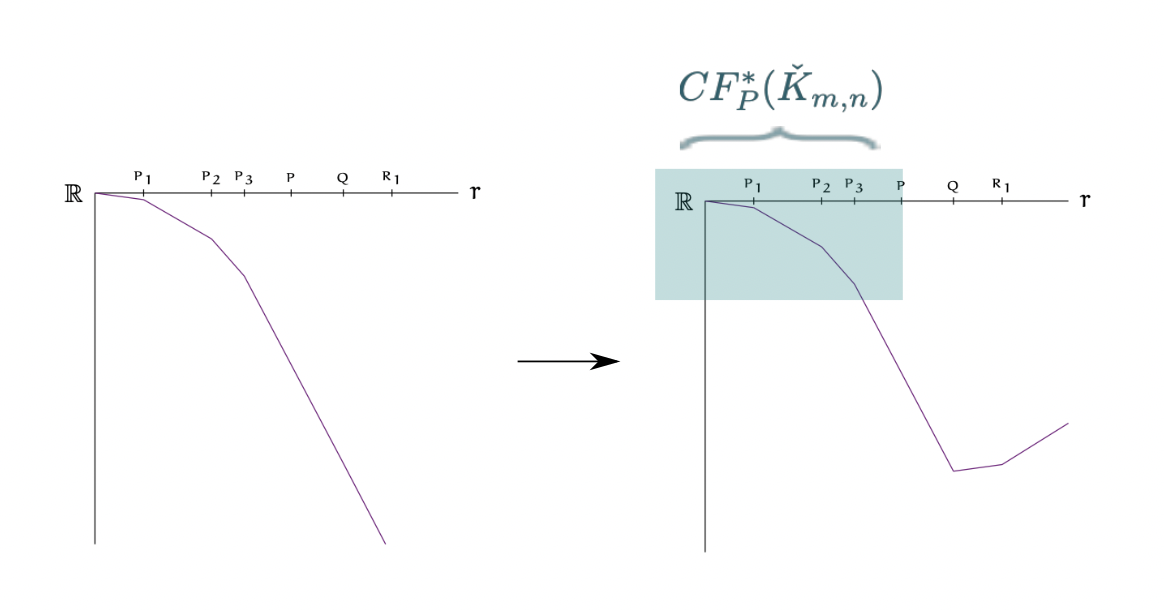}
\caption{The subcomplex $CF^*_P(\check{K}_{m, n})$, and the quasi-isomorphism of Proposition \ref{prop:qi(1)}}
\label{fig:qi(1)}
\end{figure}

\begin{remark}
Lemma \ref{lem:remain}, Lemma \ref{lem:cm0}, and Lemma \ref{lem:phiiso}, as well as Corollary \ref{cor:subcplx} and Corollary \ref{cor:phia} are completely analogous to Lemmas 1, 2, and 3
appearing in \cite{venkatesh-rab}.  The Hamiltonians $\check{K}_{m, n}$ are of a slightly different flavor, but the only thing to check is the integrated maximum principle, which we reprove in the Appendix.  With the reproven result, these Lemmas follow immediately from the results in \cite{venkatesh-rab}, and we therefore omit the proofs.  See also Lemma 3
in \cite{venkatesh-neg} for a maximum principle that is very closely related to the result we need.
\end{remark}

\begin{lemma}
\label{lem:remain}
Any Floer solution of $\check{K}_{m, n}$ or $\check{K}_{m, n}^s$ with negative asymptote contained in $E\setminus[P, \infty)\times\Sigma$ remains entirely in $E\setminus[P, \infty)\times\Sigma$.
\end{lemma}

\begin{corollary}
\label{cor:subcplx}
The following inclusions are inclusions of well-defined subcomplexes:
\[
CF^*_P(\check{K}_{m, n})\subset CF^*(\check{K}_{m, n}),
\]
\[
\lim_{\substack{\leftarrow \\ m}}CF^*_P(\check{K}_{m, n})\subset\lim_{\substack{\leftarrow \\ m}} CF^*(\check{K}_{m, n}),
\]
and
\[
\lim_{\substack{\rightarrow \\ n}}\lim_{\substack{\leftarrow \\ m}}CF^*_P(\check{K}_{m, n})\subset\lim_{\substack{\rightarrow \\ n}}\lim_{\substack{\leftarrow \\ m}} CF^*(\check{K}_{m, n}),
\]
\end{corollary}

Relate $\check{K}_{m, n}$ to $K_m$ by defining a homotopy $K_{m, 0}^s$ between the Hamiltonian $K_m$ and the Hamiltonian $\check{K}_{m, 0}$ that is monotone increasing everywhere in $s$.  Let
\[
\iota_m:CF^*(K_m)\rightarrow CF^*(\check{K}_{m, 0})
\]
be the continuation map induced by $K_{m, 0}^s$.
\begin{lemma}
\label{lem:cm0}
The map $\iota_m$ induces an isomorphism of complexes
\[
CF^*(K_m)\cong CF^*_P(K_{m, 0}).
\]
\end{lemma}
More generally, define
\[
\phi_{m, n} = c_n\circ...\circ c_0\circ \iota_m: CF^*(K_m)\rightarrow CF^*(\check{K}_{m, n}).
\]
Because each continuation map $c_n$ acts as the canonical inclusion of a subcomplex, Lemma \ref{lem:cm0} directly extends to show that $\phi_n$ induces an isomorphism of complexes
\[
CF^*(K_m)\cong CF^*_P(\check{K}_{m, n}).
\]
Let
\[
\Phi_n = \lim_{\substack{\leftarrow \\ m}}\phi_{m, n}
\]
and
\[
\Phi = \lim_{\substack{\rightarrow \\ n}} \Phi_n.
\]
\begin{lemma}
\label{lem:phiiso}
$\Phi_n$ induces an isomorphism of complexes
\[
\lim_{\substack{\leftarrow \\ m}} CF^*(K_m) \cong \lim_{\substack{\leftarrow \\ m}}CF^*_P(\check{K}_{m, n})
\]
and $\Phi$ induces an isomorphism of complexes
\[
 \lim_{\substack{\leftarrow \\ m}} CF^*(K_m) \cong  \lim_{\substack{\rightarrow \\ n}}\lim_{\substack{\leftarrow \\ m}}CF^*_P(\check{K}_{m, n}).
\]
\end{lemma}
The maps $\Phi_n$ and $\Phi$ are built out of the continuation maps $\{c_n\}$ and the inclusions $\iota_{m, n}$.  By construction, both $c_n$ and $\iota_{m, n}$ are action-preserving.  As a consequence, Lemma \ref{lem:phiiso} has an action-truncated version.  Let
\[
\Phi_{n, a}: \lim_{\substack{\leftarrow \\ m}} CF^*_{(a, \infty)}(K_m) \rightarrow \lim_{\substack{\leftarrow \\ m}}CF^*_{P, (a, \infty)}(\check{K}_{m, n})
\]
be the restriction of $\Phi_n$ to the action-truncated pieces.  Similarly define $\Phi_a$.
\begin{corollary}
\label{cor:phia}
$\Phi_{n, a}$ induces an isomorphism of complexes
\[
\lim_{\substack{\leftarrow \\ m}} CF^*_{(a, \infty)}(K_m) \cong \lim_{\substack{\leftarrow \\ m}}CF^*_{P, (a, \infty)}(\check{K}_{m, n})
\]
and $\Phi_a$ induces an isomorphism of complexes
\[
 \lim_{\substack{\leftarrow \\ m}} CF^*_{(a, \infty)}(K_m) \cong  \lim_{\substack{\rightarrow \\ n}}\lim_{\substack{\leftarrow \\ m}}CF^*_{P, (a, \infty)}(\check{K}_{m, n}).
\]
\end{corollary}
We are now ready to prove Proposition \ref{prop:qi(1)}.

\begin{proofprop1}
From Corollary \ref{cor:phia} it suffices to show that the inclusion
\[
\lim_{\substack{\rightarrow \\ n}}\lim_{\substack{\leftarrow \\ m}} CF^*_{P, (a, \infty)}(\check{K}_{m,n})\hookrightarrow\lim_{\substack{\rightarrow \\ n}}\lim_{\substack{\leftarrow \\ m}} CF^*_{(a, \infty)}(\check{K}_{m,n})
\]
is a quasi-isomorphism.  This is equivalent to showing that the quotient complex
\begin{equation} 
\label{eq:bigquot}
\bigslant{\lim\limits_{\substack{\rightarrow \\ n}}\lim\limits_{\substack{\leftarrow \\ m}} CF^*_{(a, \infty)}(\check{K}_{m,n})}{\lim\limits_{\substack{\rightarrow \\ n}}\lim\limits_{\substack{\leftarrow \\ m}} CF^*_{P, (a, \infty)}(\check{K}_{m,n})}
\end{equation}
is acyclic.  We will show that, in fact, (\ref{eq:bigquot}) is identically zero.  As direct limits commute with quotients, it suffices to show that
\begin{equation} 
\label{eq:bigquot2}
\bigslant{\lim\limits_{\substack{\leftarrow \\ m}} CF^*_{(a, \infty)}(\check{K}_{m,n})}{\lim\limits_{\substack{\leftarrow \\ m}} CF^*_{P, (a, \infty)}(\check{K}_{m,n})}
\end{equation}
is identically 0 for any $n\in\NN$.

All connecting maps are surjections, and so the Mittag-Leffler condition is satisfied.  This implies that (\ref{eq:bigquot2}) is isomorphic to the inverse limit
\begin{equation} 
\label{eq:inv}
\lim_{\substack{\leftarrow \\ m}}\hspace{.2cm}\bigslant{CF^*_{(a, \infty)}(\check{K}_{m, n})}{CF^*_{P, (a, \infty)}(\check{K}_{m, n})}.
\end{equation}
Thus, it suffices to show that
\[
0\hspace{.2cm}=\hspace{.2cm}\bigslant{CF^{\ell}_{(a, \infty)}(\check{K}_{m, n})}{CF^{\ell}_{P, (a, \infty)}(\check{K}_{m, n})}
\]
for fixed $\ell\in\RR$ and for sufficiently large $m$.  Note that we may indeed consider each graded component individually, as the inverse system respects the grading.
Define a splitting
\[
CF^{\ell}_{(a, \infty)}(\check{K}_{m, n}) = CF^{\ell}_{P, (a, \infty)}(\check{K}_{m, n})\oplus\cg{C}^{\ell}_{(a,\infty), m, n},
\]
where $\cg{C}^{\ell}_{(a,\infty), m, n}$ is generated by Novikov-weighted periodic orbits in $[P, R]\times\Sigma$ of index ${\ell}$ and action at least $a$.

Let $Cx$ be a Novikov-weighted periodic orbit in ${[P, \infty]}\times\Sigma$ of $\check{K}_{m, n}$ and of index ${\ell}$.  We want to show that $\cg{A}(Cx) < a$ for $m >> 0$.

Let $x$ have winding number $\mathfrak{w}$. As shown in, for example, \cite{venkatesh-thesis}, the grading of $Cx$ has the formula
\[
{\ell} = -2\mathfrak{w} + {2}{\kappa}ev(C) + C',
\]
where $C'$ is a universally-bounded constant.  Rearranging,
\begin{equation}
\label{eq:indexbound}
    ev(C) = \frac{1}{2\kappa}\left({\ell} + 2\mathfrak{w} - C'\right).
\end{equation}
There exists a constant $B$ independent of $m, n$ such that, if $x$ is an orbit of $\check{K}_{m, n}$ that starts in $[P, R]\times\Sigma$, 
\[
\check{K}_{m, n}(x(t)) \leq -(Q-P)m + B.
\]
The action of $Cx$ is therefore bounded above by
\begin{equation} 
\label{eq:actionbound}
\cg{A}(Cx) < -\mathfrak{w}P_{\mathfrak{w}} + ev(C) -(Q-P)m + B.
\end{equation}
Substituting (\ref{eq:indexbound}) into (\ref{eq:actionbound}) yields
\[
\cg{A}(Cx) < \mathfrak{w}\left(-P_{\mathfrak{w}} + \frac{1}{\kappa} \right) -(Q-P)m + B + \frac{1}{2\kappa}\left({\ell} - C'\right).
\]
By assumption either $P_{\mathfrak{w}} <\frac{1}{\kappa} $ and $\mathfrak{w} \leq 0$ or $P_{\mathfrak{w}} > \frac{1}{\kappa} $ and $\mathfrak{w} \geq 0$.  Either way,
\[
\mathfrak{w}\left(-P_{\mathfrak{w}} +\frac{1}{\kappa}  \right) < 0.
\]
As $m \rightarrow \infty$, $-(Q-P)m + B\rightarrow -\infty$.  Thus, there exists some value $M$ such that, for all $m > M$,
\[
\cg{A}(Cx) < a.
\]
It follows that for $m >> 0$ there is a $\Lambda_0$-module isomorphism
\[
\bigslant{CF^{\ell}_{(a, \infty)}(\check{K}_{m, n})}{CF^{\ell}_{P, (a, \infty)}(\check{K}_{m, n})} \hspace{.2cm}\cong\hspace{.2cm} \cg{C}^{\ell}_{(a, \infty), m, n} = 0.
\]
\end{proofprop1}

\begin{corollary}
There is an isomorphism
\[
\lim_{\substack{\rightarrow \\ a}} \lim_{\substack{\leftarrow \\ m}}HF^*_{(a, \infty)}(K_m) \cong \lim_{\substack{\rightarrow \\ a}}\lim_{\substack{\rightarrow \\ n}}\lim_{\substack{\leftarrow \\ m}} HF^*_{(a, \infty)}(\check{K}_{m,n})
\]
\end{corollary}

\begin{proof}
Recalling that direct limits preserve exact sequences, the quasi-isomorphism of Proposition \ref{prop:qi(1)} induces a map of Milnor exact sequences
\[
\begin{tikzcd}
0 \arrow{r} &  \lim\limits_{\substack{\leftarrow \\ m}}\text{\textsuperscript{$1$} }HF^*_{(a, \infty)}(K_m) \arrow{r}{} \arrow{d} &  H\left(\lim\limits_{\substack{\leftarrow \\ m}}CF^*_{(a, \infty)}(K_m)\right) \arrow{r} \arrow{d}{\simeq} &  \lim\limits_{\substack{\leftarrow \\ m}}HF^*_{(a, \infty)}(K_m) \arrow{r} \arrow{d}{\simeq}& 0 \\
0 \arrow{r} &  \lim\limits_{\substack{\rightarrow \\ n}}\lim\limits_{\substack{\leftarrow \\ m}}\text{\textsuperscript{$1$} } HF^*_{(a, \infty)}(\check{K}_{m,n}) \arrow{r} & H\left( \lim\limits_{\substack{\rightarrow \\ n}}\lim\limits_{\substack{\leftarrow \\ m}} CF^*_{(a, \infty)}(\check{K}_{m,n})\right) \arrow{r} &  \lim\limits_{\substack{\rightarrow \\ n}}\lim\limits_{\substack{\leftarrow \\ m}}HF^*_{(a, \infty)}(\check{K}_{m,n})\arrow{r}{} & 0
\end{tikzcd}
\]
The right-hand vertical arrow is an isomorphism because the Mittag-Leffler condition is satisfied: the continuation maps, by construction, act as canonical surjections.  Taking the direct limit $a\rightarrow-\infty$ over the right-most isomorphism yields the Corollary.

\end{proof}

\subsubsection{The isomorphisms (2) and (5)}

We begin with a general result.  Let $F_1:E\times S^1\rightarrow\RR$ and $F_2:E\times S^1\rightarrow\RR$ be two Hamiltonians with compact support in $E$ and support in $[0, 1/2]\subset S^1$.  Assume that $F_1\geq F_2$ everywhere.  Let
\[
\check{K}^{F_1, G}_{m, n} = \rho(t)(\check{k}_m(k\pi r^2) + \mathfrak{p}^*G) + F_1 \hspace{1cm}\text{and}\hspace{1cm}K^{F_2, G}_{m, n} = \rho(t)(\check{k}_m(k\pi r^2) + \mathfrak{p}^*G) + F_2.
\]
Denote by $\mathcal{S}$ the supremum of $F_1-F_2$:
\[
\mathcal{S} = \sup_{x\in E\times S^1}F_1(x)-F_2(x).
\]
The following Lemma appears in a slightly different guise in \cite{kang}.
\begin{lemma}
\label{lem:fg}
There are chain maps
\[
CF^*_{(a, \infty)}(\check{K}^{F_2, G}_{m, n})\rightarrow CF^*_{(a, \infty)}(\check{K}^{F_1, G}_{m, n} )\rightarrow CF^*_{(a, \infty)}(\check{K}^{F_2 + \mathcal{S}, G}_{m, n})\cong CF^*_{(a-\mathcal{S}, \infty)}(\check{K}^{F_2, G}_{m, n})
\]
whose composition is chain-homotopic to the inclusion of action-filtered pieces
\[
CF^*_{(a, \infty)}(\check{K}^{F_2, G}_{m, n})\hookrightarrow CF^*_{(a - \mathcal{S}, \infty)}(\check{K}^{F_2, G}_{m, n}).
\]
\end{lemma}
Kang considers Hamiltonians that are truly linear at infinity: Hamiltonians of the form $h_m(k\pi r^2)$, rather than $K_m = k_m(k\pi r^2) + (1 + k\pi r^2)\mathfrak{p}^*G$.  However, as with all other similar results in this paper, the proof of Lemma \ref{lem:fg} is identical, given that the integrated maximum principle holds.
We therefore refer to \cite{kang} for the proof of Lemma \ref{lem:fg}.

Specializing to our set-up,
\begin{lemma}
\label{lem:0fg}
There is an isomorphism
\[
\lim_{\substack{\rightarrow \\ a}}\lim_{\substack{\rightarrow \\ n}}\lim_{\substack{\leftarrow \\ m}}HF^*_{(a, \infty)}(\check{K}_{m, n}^{0, G})\rightarrow \lim_{\substack{\rightarrow \\ a}}\lim_{\substack{\rightarrow \\ n}}\lim_{\substack{\leftarrow \\ m}}HF^*_{(a, \infty)}(\check{K}_{m, n}^{F, G}).
\]
\end{lemma}

\begin{proof}
Assume without loss of generality that $F\geq 0$ everywhere.  Taking $F_2=0$, $F_1 = F$, and 
\[
\mathcal{S} = \sup_{x\in E\times S^1}F(x),
\] 
Lemma \ref{lem:fg} produces maps
\[
\mu_{m, n}: CF^*_{(a, \infty)}(\check{K}_{m, n}^{0, G})\rightarrow CF^*_{(a, \infty)}(\check{K}_{m, n}^{F, G}).
\]
and
\[
\nu_{m, n}: CF^*_{(a, \infty)}(\check{K}_{m, n}^{F, G})\rightarrow CF^*_{(a - \mathcal{S}, \infty)}(\check{K}_{m, n}^{0, G}).
\]
such that $\nu_{m, n}\circ\mu_{m, n}$ is chain-homotopic to the inclusion 
\[
\iota_{a, a-M}: CF^*_{(a, \infty)}(\check{K}_{m, n}^{0, G})\hookrightarrow CF^*_{(a - \mathcal{S}, \infty)}(\check{K}_{m, n}^{0, G}).
\]
Let $\mu_{m, n}^*$ and $\nu_{m, n}^*$ be the induced maps on homology.  As continuation maps commute up to homology, there are induced maps on limits
\[
M^*_a = \lim_{\substack{\rightarrow \\ n}} \lim_{\substack{\leftarrow \\ m}}\mu_{m, n}^*: \lim_{\substack{\rightarrow \\ n}}\lim_{\substack{\leftarrow \\ m}}HF^*_{(a, \infty)}(\check{K}_{m, n}^{0, G})\rightarrow \lim_{\substack{\rightarrow \\ n}}\lim_{\substack{\leftarrow \\ m}}HF^*_{(a, \infty)}(\check{K}_{m, n}^{F, G})
\]
and
\[
N^*_a = \lim_{\substack{\rightarrow \\ n}} \lim_{\substack{\leftarrow \\ m}}\nu_{m, n}^*:\lim_{\substack{\rightarrow \\ n}}\lim_{\substack{\leftarrow \\ m}}HF^*_{(a, \infty)}(\check{K}_{m, n}^{F, G})\rightarrow \lim_{\substack{\rightarrow \\ n}}\lim_{\substack{\leftarrow \\ m}}HF^*_{(a-\mathcal{S}, \infty)}(\check{K}_{m, n}^{0, G})
\]
such that
\[
N^*_{a}\circ M^*_a = \iota^*_{a, a-\mathcal{S}}
\]
is the inclusion of action-filtered subcomplexes.
Taking the limit over action window $(a, \infty)$ produces maps $M^*$ and $N^*$ with
\begin{equation} 
\label{eq:ninvmoneway}
N^*\circ M^* = \lim_{\substack{\rightarrow \\a}}\iota^*_{a, a-\mathcal{S}} = id.
\end{equation}

Conversely, there are maps
\[
\cg{M}^*_a: \lim_{\substack{\rightarrow \\ n}}\lim_{\substack{\leftarrow \\ m}}HF^*_{(a, \infty)}(\check{K}_{m, n}^{\mathcal{S}, G})\rightarrow \lim_{\substack{\rightarrow \\ n}}\lim_{\substack{\leftarrow \\ m}}HF^*_{(a-\mathcal{S}, \infty)}(\check{K}_{m, n}^{F, G})
\]
and
\[
\cg{N}^*_a: \lim_{\substack{\rightarrow \\ n}}\lim_{\substack{\leftarrow \\ m}}HF^*_{(a, \infty)}(\check{K}_{m, n}^{F, G})\rightarrow \lim_{\substack{\rightarrow \\ n}}\lim_{\substack{\leftarrow \\ m}}HF^*_{(a, \infty)}(\check{K}_{m, n}^{\mathcal{S}, G})
\]
satisfying
\begin{equation} 
\label{eq:mniota}
\cg{M}^*_a\circ\cg{N}^*_a = \iota_{a, a-\mathcal{S}}.
\end{equation}
The identification $CF^*_{(a, \infty)}(\check{K}_{m, n}^{M, G}) = CF^*_{(a-\mathcal{S}, \infty)}(\check{K}_{m, n}^{0, G})$ identifies the maps $\cg{M}_a$ and $M_{a-\mathcal{S}}$ and the maps $N_a^*$ and $\cg{N}_{a-\mathcal{S}}^*$.  Under this identification, Equation (\ref{eq:mniota}) becomes
\[
M_{a-M}^*\circ N_a^* = \iota_{a, a-\mathcal{S}}.
\]
Taking the limit over action window $(a, \infty)$ yields the identity
\begin{equation} 
\label{eq:ninvmotherway}
M^*\circ N^* = id.
\end{equation}
Equations (\ref{eq:ninvmoneway}) and (\ref{eq:ninvmotherway}) shows that $M^*$ and $N^*$ are inverses.

\end{proof}

Note that the Hamiltonian $\check{K}_{m, n}$ is precisely the Hamiltonian $\check{K}_{m, n}^{0, G}$.  Thus, the isomorphism (2) is shown by the following Proposition.

\begin{proposition}
\label{prop:qi(2)}
There is an isomorphism 
\[
\lim_{\substack{\rightarrow \\ a}}\lim_{\substack{\rightarrow \\ n}}\lim_{\substack{\leftarrow \\ m}}HF^*_{(a, \infty)}(\check{K}_{m, n}^{0, G})\rightarrow \lim_{\substack{\rightarrow \\ a}}\lim_{\substack{\rightarrow \\ n}}\lim_{\substack{\leftarrow \\ m}}HF^*_{(a, \infty)}(\check{K}_{m, n}^{F}).
\]
\end{proposition}

\begin{proof}
By Lemma \ref{lem:0fg} it suffices to show that there is an isomorphism
\[
\lim_{\substack{\rightarrow \\ a}}\lim_{\substack{\rightarrow \\ n}}\lim_{\substack{\leftarrow \\ m}}HF^*_{(a, \infty)}(\check{K}_{m, n}^{F, G})\rightarrow \lim_{\substack{\rightarrow \\ a}}\lim_{\substack{\rightarrow \\ n}}\lim_{\substack{\leftarrow \\ m}}HF^*_{(a, \infty)}(\check{K}_{m, n}^{F}).
\]
The proof of this isomorphism is very similar to the proof of Lemma \ref{lem:0fg}.  We therefore outline the strategy and refer to the above proof for details.

Assume without loss of generality that $G \geq 0$.  Let
\[
\mathcal{S} = \max_{x\in M} G(x)
\]
and
\[
\mathcal{R} = \min_{\substack{ \text{supp}(F)\subset D_r}} r
\]
Let $G_s^1$ be a monotone homotopy such that $G_s^1 = 0$ when $s >> 0$ and $G_s^1 = G$ when $s << 0$.  Let $G_s^2$ be a monotone homotopy such that $G_s^2 = G$ when $s >> 0$ and $G_s^2 = \mathcal{S}$ when $s << 0$.  As in the proof of Lemma \ref{lem:0fg}, the Hamiltonians 
\[
\rho(t)(\check{k}_{m, n}(k\pi r^2) + (1 + k\pi r^2)\mathfrak{p}^*G_s^1) + F\hspace{1cm}\text{and}\hspace{1cm} \rho(t)(\check{k}_{m, n}(k\pi r^2) + (1 + k\pi r^2)\mathfrak{p}^*G_s^2) + F
\]
induce maps
\[
CF^*_{(a, \infty)}(\check{K}_{m, n}^{F, 0}) \rightarrow CF^*_{(a, \infty})(\check{K}_{m, n}^{F, G}) \rightarrow CF^*_{(a - (1 + k\pi \mathcal{R}^2)\mathcal{S}, \infty})(\check{K}_{m, n}^{F, 0}).
\]
whose composition is homotopic to the inclusion map
\[
CF^*_{(a, \infty)}(\check{K}_{m, n}^{F, 0}) \hookrightarrow CF^*_{(a - (1 +k\pi \mathcal{R}^2)\mathcal{S}, \infty})(\check{K}_{m, n}^{F, 0}).
\]
Note that compactness of Floer trajectories is achieved by the integrated maximum principle.  Similarly, all periodic orbits remain inside the disk bundle of radius $k\pi \mathcal{R}^2$; thus, the maximum bound on a change in action-window is given by $(1 + k\pi \mathcal{R}^2)\mathcal{S}$.

Similarly, there are maps
\[
CF^*_{(a, \infty)}(\check{K}_{m, n}^{F, G}) \rightarrow CF^*_{(a, \infty})(\check{K}_{m, n}^{F, \mathcal{S}}) \rightarrow CF^*_{(a - (1 + k\pi \mathcal{R}^2)\mathcal{S}, \infty})(\check{K}_{m, n}^{F, G}).
\]
whose composition is homotopic to the inclusion map
\[
CF^*_{(a, \infty)}(\check{K}_{m, n}^{F, G}) \hookrightarrow CF^*_{(a - (1 + k\pi \mathcal{R}^2)\mathcal{S}, \infty})(\check{K}_{m, n}^{F, G}).
\]
As in the proof of Lemma \ref{lem:0fg}, taking homology and limits over $m, n$ and $a$ yields inverse maps
\[
\lim_{\substack{\rightarrow \\ a}}\lim_{\substack{\rightarrow \\ n}}\lim_{\substack{\leftarrow \\ m}}HF^*_{(a, \infty)}(\check{K}_{m, n}^{F, G})\rightleftharpoons \lim_{\substack{\rightarrow \\ a}}\lim_{\substack{\rightarrow \\ n}}\lim_{\substack{\leftarrow \\ m}}HF^*_{(a, \infty)}(\check{K}_{m, n}^{F}).
\]
The Proposition follows.

\end{proof}

The isomorphism (5) proceeds in the same manner.

\begin{proposition}
\label{prop:qi(5)}
There is an isomorphism
\[
\lim_{\substack{\rightarrow \\ a}}\lim_{\substack{\rightarrow \\ n}}\lim_{\substack{\leftarrow \\ m}} HF^*_{(a, \infty)}(H_n^F)\cong \lim_{\substack{\rightarrow \\ a}}\lim_{\substack{\rightarrow \\ n}}\lim_{\substack{\leftarrow \\ m}} HF^*_{(a, \infty)}(H_n).
\]
\end{proposition}

\begin{proof}
The proof is completely analogous to the proof of Proposition \ref{prop:qi(2)}, and so we omit it.

\end{proof}

\subsubsection{The isomorphisms (3) and (4)}

We prove the following Propositions.

\begin{proposition}
\label{prop:qi(3)}
There is an isomorphism
\[
\lim_{\substack{\rightarrow \\ a}}\lim_{\substack{\rightarrow \\ n}}\lim_{\substack{\leftarrow \\ m}} HF^*_{(a, \infty)}(\check{K}_{m,n}^F)\simeq \lim_{\substack{\rightarrow \\ a}}\lim_{\substack{\rightarrow \\ n}}\lim_{\substack{\leftarrow \\ m}} HF^*_{(a, \infty)}(\check{H}_{m,n}^F)
\]
\end{proposition}

\begin{proposition}
\label{prop:qi(4)}
There is an isomorphism
\[
\lim_{\substack{\rightarrow \\ a}}\lim_{\substack{\rightarrow \\ n}}\lim_{\substack{\leftarrow \\ m}}HF^*_{(a, \infty)}(\check{H}_{m, n}^F)\simeq
\lim_{\substack{\rightarrow \\ a}}\lim_{\substack{\rightarrow \\ n}} HF^*_{(a, \infty)}(H_n^F).
\]
\end{proposition}

These propositions require a careful analysis of the action of orbits on the conical end of $E$, which we package into Lemma \ref{lem:bij} and Lemma \ref{lem:actionest}.

\begin{lemma}
\label{lem:bij}
Suppose $m > 0$.  Away from the collar neighborhood $[P_n, P_{n+1}]\times\Sigma$ there is a set bijection between periodic orbits of $K_m^F$ and periodic orbits of $K_{m+1}^F$.  

Suppose $n \geq 0$.  Away from the collar neighborhood $[R_n, R_{n+1}]\times\Sigma$ there is a set bijection between periodic orbits of $H_n^F$ and periodic orbits of $H_{n+1}^F$.  
\end{lemma}

\begin{proof}
We prove the Lemma for $n \geq 0$.  The case $m > 0$ is completely analogous.

The Hamiltonians vector fields of $H_n^F$ and $H_{n+1}^F$ coincide on $\left(E\setminus[R_n, \infty)\times\Sigma\right)\times [0,\frac{1}{2}]$ and on $E\times[\frac{1}{2}, 1]$.  Any periodic orbit that starts in $E\setminus[R_n, \infty)\times\Sigma$ remains in $E\setminus[R_n, \infty)\times\Sigma$ through the time interval $[0, \frac{1}{2}]$.  Thus, there is a set bijection between the subset of orbits starting in $E\setminus[R_n, \infty)\times\Sigma$.

The integral curves corresponding to ${H_n^F}$, respectively ${H_{n+1}^F}$, on $[R_{n+1}, \infty)\times\Sigma\times [0, \frac{1}{2}]$ is fiberwise $re^{2\pi (i + 1/2) nt}$, respectively $re^{2\pi i(n+3/2)t}$.  If $x_n(t)$ is an integral curve of ${H_n^F}$ and $x_{n+1}(t)$ is an integral curve of ${H_{n+1}^F}$ with $x_n(0) = x_{n+1}(0) \in [R_{n+1}, \infty)\times\Sigma$ then it follows that
\[
x_n(1/2) = x_{n+1}(1/2).
\]
As the Hamiltonian vector fields of ${H_n^F}$ and ${H_{n+1}^F}$ coincide on $E\times[\frac{1}{2}, 1]$, it follows that $x_n(t) = x_{n+1}(t)$ for all $t\in[\frac{1}{2}, 1]$.  In particular, $x_n(1) = x_{n+1}(1)$.  Thus, $x_n$ is a periodic orbit if and only if $x_{n+1}$ is a periodic orbit.

\end{proof}

%%%change \kappas to ks?

\begin{lemma}
\label{lem:actionest}
Fix an action window $(a, \infty)$ and a degree $q$.  For $n >> 0$ large enough, all periodic orbits appearing in $CF^q_{(a, \infty)}({H_n^F})$ begin in $E\setminus [R_n, \infty)\times\Sigma$.  Similarly, for $m << 0$ negative enough, all periodic orbits appearing in $CF^q_{(a, \infty)}(K_m^F)$ begin in $E\setminus[P_m, \infty)\times\Sigma$.
\end{lemma}

\begin{proof}
Let $x(t)$ be a periodic orbit of $F + \rho(t)\frac{1}{2}k\pi r^2$.  With the orientation induced from the complex structure, let $z$ be the arc segment starting at $x(1/2)$ and continuing to $x(1)$.  Fix a capping disk $\tilde{v}_x$ of the loop $\ell = (-z)\#x\big|_{[1/2, 1]}$.
\begin{figure}[htbp!]
\centering
\includegraphics[scale=.5]{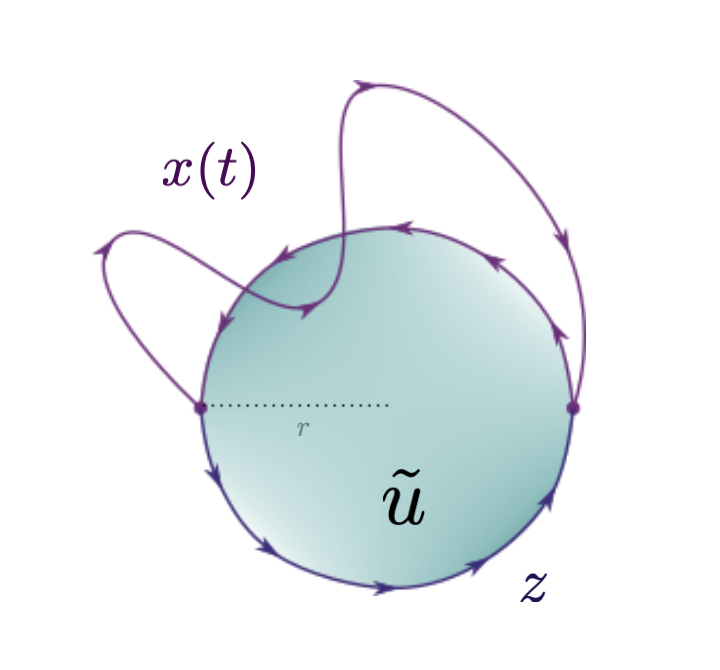}
\caption{The capping disk in the fiber}
\label{fig:capping}
\end{figure}

Let $\tilde{u}$ be the fiber disk of the loop $\gamma = x\big|_{[0, 1/2]}\#z$.  See Figure \ref{fig:capping}.  Then $x \simeq \gamma\#\ell$ has a capping disk $\tilde{x} \simeq \tilde{u}\#\tilde{v}$ such that
\[
\int_{D}\tilde{x}^*\Omega = \int_D(\tilde{u}\#\tilde{v})^*\Omega = -2k\pi r(x)^2 + \int_D\tilde{v}^*\Omega,
\]
where $r(x)$ is the radius at which $x\big|_{[0, 1/2]}$ lives.  The Conley-Zehnder index $\mu_{CZ}$ of $\tilde{x}$ is, by the Additivity Axiom,
\[
\mu_{CZ}(\tilde{x}) = \mu_{CZ}(\tilde{u})+\mu_{CZ}(\tilde{v}) = -2 + \mu_{CZ}(\tilde{v})
\]
(see, for example, \cite{salamon}).

Now let $n \geq 0$, and let $x_n(t)$ be a periodic orbit of $H_n^F$ living in $[R_n, \infty)\times\Sigma$.  By definition of ${H_n^F}$, $x_n(t)$ is a periodic orbit of $F+ \rho(t)(nk\pi r^2 + \frac{1}{2}k\pi r^2)$.  Thus, by Lemma \ref{lem:bij}, $x_n(t)$ corresponds to some periodic orbit $x(t)$ of $F + \rho(t)\frac{1}{2}k\pi r^2$.  Note that $x$ and $x_n$ begin at the same point, and therefore begin at the same radius $r(x)$; similarly, $x\big|_{[1/2, 1]} = x_n\big|_{[1/2, 1]}$.  As above, let $z$ be the arc segment starting at $x(1/2)$ and continuing widdershins to $x(1)$ and let $\tilde{v}$ be the fixed capping disk of $(-z)\#x\big|_{[1/2, 1]} = (-z)\#x_n\big|_{[1/2, 1]}$.  Let $\tilde{u}_n$ be the fiber disk of the loop $x_n\big|_{[0, 1/2]}\#z$, so that $x_n$ has a capping disk $\tilde{x}_n\simeq \tilde{u}_n\#\tilde{v}$, satisfying
\[
\int_{D}\tilde{x}_n^*\Omega = \int_D(\tilde{u}_n\#\tilde{v})^*\Omega = -2k\pi nr(x)^2 + \int_D\tilde{v}^*\Omega
\]
Note that 
\[
\mu_{CZ}(\tilde{x}_n) = \mu_{CZ}(\tilde{u}_n) + \mu_{CZ}(\tilde{v}) = -2n + \mu_{CZ}(\tilde{v}).
\]
Away from the support of $F$, the function $F+ \rho(t)\frac{1}{2}\kappa\pi r^2$ has flow $re^{(1/2)\pi i t}$, which has no time-one orbits.  All periodic orbits are therefore contained in a compact set.  By the non-degeneracy assumption, there are finitely many periodic orbits.  The symplectic area of each chosen cap $\tilde{v}$ is bounded by some constant $C_{\Omega}$:
\[
\bigg|\int_D\tilde{v}^*\Omega\bigg| \leq C_{\Omega}.
\]
Similarly, the Conley-Zehnder index of each chosen cap $\tilde{v}$ is bounded by some constant $C_{\mu}$:
\[
|\mu_{CZ}(\tilde{v})| \leq C_{\mu}.
\]

Now let $X = T^{\alpha}x_n(t)$ be the orbit $x_n$ with a Novikov weight, and suppose $X$ lies in degree $q$.  As shown in, for example, \cite{venkatesh-thesis}, $X$ has grading
\begin{equation}
\label{eq:grading}
q = \mu(X) = 2\kappa\alpha - 2n + \mu_{CZ}(\tilde{v}).
\end{equation}
The bounds on $\mu_{CZ}(\tilde{v})$ imply that
\[
q - C_{\mu} \leq 2\kappa\alpha - 2n \leq q + C_{\mu},
\]
so that
\begin{equation}
\label{eq:indexineq}
\frac{1}{2\kappa}\left(q - C_{\mu} + 2n\right) \leq \alpha \leq \frac{1}{2\kappa}\left(q + C_{\mu} + 2n\right)
\end{equation}
The action of $X$ is, by definition,
\[
\cg{A}(X) = -k\pi (n+\frac{1}{2})r(x)^2 + \alpha + \int_D\tilde{v}^*\Omega + \int_0^1H_n^F(x_n)dt
\]

Let $(a, \infty)$ be a fixed action window.  $X$ is in $CF^{\ell}_{(a, \infty)}({H_n^F})$ if and only if
\[
a < \alpha -k\pi (n+\frac{1}{2})r(x)^2 + \int_D\tilde{v}^*\Omega+ \int_0^1H_n^F(x_n)dt,
\]
for which
\[
a - C_{\Omega} < \alpha - k\pi (n+\frac{1}{2}) r(x)^2 + \int_0^1H_n^F(x_n)dt 
\]
must be true.  Let
\[
C_F = \max_{p\in E}|{F}(p)|.
\]
There exists a bound $B$, independent of $n$, such that
\[
\int_0^{1/2}{H_n^F}(k\pi r(x(t))^2)dt < (n+\frac{1}{2})\left(k\pi r(x)^2-R_n\right) + B.
\]
Altogether,
\[
\int_0^1 H_n^F(x_n)dt \leq C_F + (n+\frac{1}{2})k\pi r(x)^2 - (n + \frac{1}{2})R_n + B.
\]
Thus,
\begin{equation}
\label{eq:actionineq}
a - C_{\Omega} - C_F  - B + (n+\frac{1}{2})R_n  < \alpha
\end{equation}

Suppose $n >> 0$.  More precisely, suppose that $n$ is large enough that
\[
\left(R_n - \frac{1}{\kappa}\right)n > -a + C_{\Omega} + C_F + B + \frac{q}{2\kappa} - \frac{C_{\mu}}{2\kappa}- \frac{1}{2}R_n.
\]
Note that $R_n - \frac{1}{\kappa}$ is bounded below by some positive number by construction, so this is always possible.
Combining equations (\ref{eq:indexineq}) and (\ref{eq:actionineq}) yields a contradiction:
\[
\alpha\leq \frac{1}{2\kappa}\left(q - C_{\mu} + 2n\right) < a - C_{\Omega} - C_F -B  + (n+\frac{1}{2})R_n < \alpha.
\]
Now let $m < 0$, and suppose that $x_m(t)$ is a periodic orbit of $K_m^F$ living in $[P_n, \infty)\times\Sigma$.  Running the same program, suppose that $m$ is negative enough that
\[
\left(P_m - \frac{1}{\kappa}\right)m > -a + C_{\Omega} + C_F + B + \frac{q}{2\kappa} - \frac{C_{\mu}}{2\kappa}- \frac{1}{2}P_m.
\]
Note that $P_m - \frac{1}{\kappa}$ is bounded above by some negative number by construction, so this is always possible.  Combining equations (\ref{eq:indexineq}) and (\ref{eq:actionineq}) again yields a contradiction:
\[
\alpha\leq \frac{1}{2\kappa}\left(q - C_{\mu} + 2m\right) < a - C_{\Omega} - C_F - B + (m+\frac{1}{2})P_m) < \alpha.
\]
The Lemma follows.

\end{proof}

The proofs of Proposition \ref{prop:qi(3)} and Proposition \ref{prop:qi(4)} now follow as corollaries of two Lemmas.  To state the first Lemma, let $C_m = \check{H}_{m, n} - \check{K}_{m, n}$.

\begin{lemma}
\label{lem:km}
There is a quasi-isomorphism
\[
\lim_{\substack{\leftarrow \\ a}}\lim_{\substack{\rightarrow \\ n}}\lim_{\substack{\leftarrow \\ m}} CF^*_{(a, \infty)}(\check{K}_{m,n}^F)\cong \lim_{\substack{\leftarrow \\ a}}\lim_{\substack{\rightarrow \\ n}}\lim_{\substack{\leftarrow \\ m}} CF^*_{(a+C_m, \infty)}(\check{H}_{m,n}^{F}).
\]
\end{lemma}

\begin{proof}
By construction,
\begin{align*}
\check{H}_{m, n}^F &= \rho(t)\check{H}_{m, n} + F(t) \\
&= \rho(t)\left(\check{K}_{m, n} + C_m\right) + F(t).
\end{align*}
The periodic orbits of $\check{H}_{m, n}^F$ and $\check{K}_{m, n}^F$ coincide.  Let $x(t)$ be one such periodic orbit.  The action of $x(t)$ as a periodic orbit of $\check{H}_{m, n}^F$ is
\begin{align*}
\cg{A}_{\check{H}_{m, n}^F}(x) &= -\int_{D}\tilde{x}^*\Omega + \int_0^1\check{H}_{m, n}^F(x(t))dt \\
&= -\int_{D}\tilde{x}^*\Omega + \int_0^1\rho(t)\check{H}_{m, n}(x(t)) + F(x(t)) dt \\
&=-\int_{D}\tilde{x}^*\Omega + \int_0^1\rho(t)\left(\check{K}_{m, n} + C_m\right)(x(t)) + F(x(t))dt \\
&= -\int_{D}\tilde{x}^*\Omega + \int_0^1\rho(t)\check{K}_{m, n}(x(t)) + F(x(t))dt + C_m \\
&= \cg{A}_{\check{K}_{m, n}^F}(x) + C_m.
\end{align*}
The Lemma follows.

\end{proof}

\begin{lemma}
\label{lem:hh}
Suppose that $F$ has no leafwise intersection point in $[P, R]\times \Sigma$.  Then there is a quasi-isomorphism
\[
\lim_{\substack{\leftarrow \\ a}}\lim_{\substack{\rightarrow \\ n}}\lim_{\substack{\leftarrow \\ m}} CF^*_{(a+C_m, \infty)}(\check{H}_{m,n}^F) \cong \lim_{\substack{\leftarrow \\ a}}\lim_{\substack{\rightarrow \\ n}}\lim_{\substack{\leftarrow \\ m}} CF^*_{(a, \infty)}(\check{H}_{m,n}^F)
\]
\end{lemma}

\begin{proof}
Fix $\ell\in \RR$.  It suffices to show that there is an isomorphism
\begin{equation} 
\label{eq:h=k}
CF^{\ell}_{(a + C_m, \infty)}(\check{H}_{m, n}^F)\cong CF^{\ell}_{(a, \infty)}(\check{H}_{m, n}^F)
\end{equation}
for large enough values of $n$ and negative enough values of $a$.  Equation (\ref{eq:h=k}) will hold on the level of cochain complexes if it holds on the level of sets.  Thus, it suffices to show that if $X\in CF^{\ell}(\check{H}^F_{m, n})$ has action $\cg{A}(X) > a$, the action of $X$ also satisfies $\cg{A}(X) > a + C_m$.

As $\check{H}_{m, n}^F = \check{H}_{m', n'}^F$ on $E\setminus[P_1, \infty)\times\Sigma$, there is a fixed finite set
\[
\left\{C_1x_1, \dots, C_Nx_N\right\}
\]
to which the generators of $CF^{\ell}_{(a+C_m, \infty)}(\check{H}_{m, m})$ belong, for any $m, n$.  Let
\[
C = \min_{1\leq i\leq N}\cg{A}_{\check{H}^F_{-1, 0}}(C_ix_i)
\]
and choose $a < C$.

By Lemma \ref{lem:actionest}, if $n$ is large enough, all periodic orbits of $\check{H}_{m, n}^F$ with action greater than $a$ are contained inside $E\setminus[R, \infty)\times\Sigma$.  By assumption, this means that all periodic orbits are contained inside $E\setminus[P_1, \infty)\times\Sigma$.
We may therefore view $X$ as an element of $CF^{\ell}(\check{H}_{-1, n}^F)$ with action
\[
\cg{A}_{\check{H}_{-1, n}^F}(X) \geq C > a.
\]
Since $\check{H}_{m, n}^F = \check{H}_{-1, n}^F + C_m$,
\begin{align*} 
\cg{A}_{\check{H}^F_{m, n}}(X) = \cg{A}_{\check{H}^F_{-1, n}}(X) + C_m > a + C_m,
\end{align*}
and the claim follows.

\end{proof}

%Proof of PROPOSITION \ref{prop:qi(3)}
\begin{proofprop4}
As in the proof of Proposition \ref{prop:qi(1)}, Lemma \ref{lem:km} implies that there is an isomorphism
\[
\lim_{\substack{\leftarrow \\ a}}\lim_{\substack{\rightarrow \\ n}}\lim_{\substack{\leftarrow \\ m}} HF^*_{(a, \infty)}(\check{K}_{m,n}^F)\cong \lim_{\substack{\leftarrow \\ a}}\lim_{\substack{\rightarrow \\ n}}\lim_{\substack{\leftarrow \\ m}} HF^*_{(a+C_m, \infty)}(\check{H}_{m,n}^{F}).
\]
Lemma \ref{lem:hh} provides an isomorphism
\[
\lim_{\substack{\leftarrow \\ a}}\lim_{\substack{\rightarrow \\ n}}\lim_{\substack{\leftarrow \\ m}} CF^*_{(a+C_m, \infty)}(\check{H}_{m,n}^F) \cong \lim_{\substack{\leftarrow \\ a}}\lim_{\substack{\rightarrow \\ n}}\lim_{\substack{\leftarrow \\ m}} CF^*_{(a, \infty)}(\check{H}_{m,n}^F).
\]
The Mittag-Leffler solution is satisfied because the dimension of $CF^{\ell}_{(a, \infty)}(\check{H}_{m,n}^F)$ is degree-wise universally bounded.  Thus, there is an isomorphism
\[
\lim_{\substack{\leftarrow \\ a}}\lim_{\substack{\rightarrow \\ n}}\lim_{\substack{\leftarrow \\ m}} HF^*_{(a+C_m, \infty)}(\check{H}_{m,n}^F) \cong \lim_{\substack{\leftarrow \\ a}}\lim_{\substack{\rightarrow \\ n}}\lim_{\substack{\leftarrow \\ m}} HF^*_{(a, \infty)}(\check{H}_{m,n}^F).
\]
Stringing these isomorphisms together yields the Proposition.

\end{proofprop4}

%proof [OF PROPOSITION \ref{prop:qi(4)}]
\begin{proofprop5}
By the assumption that there are no leaf-wise intersection points on a circle subbundle within $[P_1, R]\times\Sigma$,  the periodic orbits of any $\check{H}_{m, n}^F$ are naturally divided into:
\begin{enumerate}
\item a finite set $\{x_1, \dots, x_N\}$ of orbits that begin inside $E\setminus[P_1, \infty)\times\Sigma$ that is the same for every choice of $m, n$, and
\item orbits that begin in $[R, \infty)\times\Sigma$.
\end{enumerate}
Fix $\ell\in\RR$.  By Lemma \ref{lem:actionest}, for sufficiently large $n$ and any $m$, all summands of an element in $CF^{\ell}_{(a, \infty)}(\check{H}_{m, n}^F)$ are of type (1).  Let $C_1, \dots C_N$ be Novikov coefficients such that $|C_ix_I| = \ell$.  Since 
\[
\check{H}_{m, n}^F\big|_{E\setminus[P_1, \infty)\times\Sigma} = \check{H}_{0, n}^F\big|_{E\setminus[P_1, \infty)\times\Sigma}
\] 
for all $n, n'$, $C_ix_i$ can be viewed as an element of $CF^{\ell}(\check{H}_{0, n}^F)$ (of smaller action).
Fix any
\begin{equation}
\label{eq:boundona}
a < \min_{1\leq i\leq N} \cg{A}_{\check{H}_{0, n}^F}(C_ix_i) = \min_{\substack{1\leq i\leq N \\ m\in\ZZ_{<0}}} \cg{A}_{\check{H}_{m, n}^F}(C_ix_i) 
\end{equation}
Fix $m\in\ZZ_{< 0}$.  Since all elements have action greater than $a$, the quasi-inverse maps
\[
CF^{\ell}(\check{H}_{m, n}^F) \xleftrightharpoons{\rule{.5cm}{0pt}}CF^{\ell}(\check{H}_{0, n}^F)
\]
restrict to quasi-isomorphisms
\[
CF^{\ell}_{(a, \infty)}(\check{H}_{m, n}^F) \xleftrightharpoons{\rule{.5cm}{0pt}} CF^{\ell}_{(a, \infty)}(\check{H}_{0, n}^F)
\]
We conclude that there are isomorphisms
\[
HF^{\ell}_{(a, \infty)}(\check{H}_{m, n}^F) \simeq HF^{\ell}_{(a, \infty)}(\check{H}_{0, n}^F) 
\]
compatible with continuation.

\end{proofprop5}

\subsection{Vanishing computations}
\label{subsec:(non)vanish}

\subsubsection{Proof that $\widehat{SH_*}(K) = 0$}
\begin{proposition}
\label{prop:shvanishes}
The completed symplectic homology of the family $\{K_m\}$ vanishes:
\[
\widehat{SH_*}(K) = 0.
\]
\end{proposition}

The proof of Proposition \ref{prop:shvanishes} closely follows the work of Albers-Kang \cite{albers-kang}.
The tangent bundle of $E$ has a pointwise splitting $T_xE = T_x\CC\oplus T_xB$.  Let $\cg{J}$ be the set of $S^1$-families of almost-complex structures that are $\Omega$-tame and upper-triangular with respect to this splitting.  Thus, a $J\in\cg{J}$ looks like
\[
J = \left[ \begin{array}{cc} i_t & B \\ 0 & j\end{array}\right].
\]
We state without proof a few Lemmas, referring to \cite{albers-kang} and \cite{venkatesh-thesis} instead.
\begin{lemma}
The set of $J\in\cg{J}$ that are regular, and whose components $i_t$ and $j$ are regular as well, is comeager and non-empty.  Gromov compactness holds.
\end{lemma}
\begin{lemma}
\label{lem:filtered}
The complex $\lim\limits_{\substack{\leftarrow \\ m}} CF^*_{(a, \infty)}(K_m)$ has a filtration given by $\mu_{CZ}(\mathfrak{p}_*(x))$.  In other words, the projection $\mathfrak{p}_*(x)$ corresponds to a periodic orbit of $G$, the projection $\mathfrak{p}_*(\tilde{x})$ corresponds to a capping of $\mathfrak{p}_*(x)$, and the Conley-Zehnder index of $\mathfrak{p}_*(\tilde{x})$ lifts to a filtration on $\lim\limits_{\substack{\leftarrow \\ m}} CF^*(K_m)$.
\end{lemma}
\begin{lemma}
\label{lem:filtered-d}
A Floer trajectory $u$ of $(K_m, J)$ descends to a Floer trajectory $\mathfrak{p}_*u$ of $(G, j)$.  The Floer differential decomposes as 
\[
\dd^{fl} = \dd_0 + \dd_1 + \dd_2 + \dots 
\]
where $\dd_{\ell}$ counts Floer trajectories $u$ such that $\mathfrak{p}_*u$ sits in a $k$-dimensional strata of the moduli space of Floer trajectories of $(G, j)$.  The decomposition of $\dd^{fl}$ satisfies
\[
(\dd_0 + \dd_1 + \dots + \dd_k)^2 = 0
\]
for all $k\geq 0$.
\end{lemma}

\begin{lemma}
\label{lem:dd0}
If $\dd_0(X) = 0$, then there exists $Y\in\lim\limits_{\substack{\leftarrow \\ m}} CF^*(K_m)$ such that $\dd_0(Y) = X$.  The action of $Y$ is 
\[
\cg{A}(Y) \geq \cg{A}(X) - A.
\]
\end{lemma}

\begin{lemma}
\label{lem:czaction}
A sequence $\{X_1, X_2, X_3, ...\}$ of elements in $\lim\limits_{\substack{\leftarrow \\ m}} CF^*(K_m)$ satisfies
\[
\mu_{CZ}(\mathfrak{p}_*(X_i))\rightarrow\pm\infty 
\]
if and only if
\[
\cg{A}(X_i)\rightarrow\pm\infty.
\]
\end{lemma}
We mimic the proof of Theorem 1 in \cite{albers-kang}

%Proof of prop 6
\begin{proofprop6}
Let $X\in \lim\limits_{\substack{\substack{\rightarrow \\ a}}}\lim\limits_{\substack{\leftarrow \\ m}} CF^*_{(a, \infty)}(K_m)$ be a cocycle.  By design the continuation maps $c:CF^*(K_m)\rightarrow CF^*(K_{m-1})$ are canonical surjections.  Because of this, $X$ can be written as a sum of weighted periodic orbits
\[
X = \sum_{i=0}^{\infty}C_ix_i,
\]
where $C_i\in\Lambda$, $x_i\in\cg{P}(K)$, and $\lim\limits_{i\rightarrow\infty}\cg{A}(C_ix_i)\neq -\infty$.  Lemma \ref{lem:czaction} states that there exists a lower bound $P$ such that 
\[
\mu_{CZ}(C_i\mathfrak{p}_*(x_i)) \geq P
\]
for all $i$.  Let
\[
X_p = \sum_{\mu_{CZ}(C_i\mathfrak{p}_*(x_i)) = p} C_ix_i.
\]
Suppose there exists $Y\in \lim\limits_{\substack{\substack{\rightarrow \\ a}}}\lim\limits_{\substack{\leftarrow \\ m}} CF^*_{(a, \infty)}(K_m)$, decomposed into filtered pieces as
\[
Y = \sum_{p=P}^{\infty}Y_p,
\]
and satisfying
\begin{equation} 
\label{eq:inductive}
\dd_0(Y_p) = -\dd_1(Y_{p-1}) - \dd_2(Y_{p-2}) - \dots - \dd_{p-P}(Y_P) + X_p.
\end{equation}
Then
\[
X_p = \dd_0(Y_p) + \dd_1(Y_{p-1}) + \dd_2(Y_{p-2}) + \dots + \dd_{p-P}(Y_P),
\]
\[
X = \sum_{p=P}^{\infty} X_p = \sum_{p=P}^{\infty} \dd_0(Y_p) + \dd_1(Y_{p-1}) + \dd_2(Y_{p-2}) + \dots + \dd_{p-P}(Y_P),
\]
and so, rearranging,
\[
X = \sum_{p=P}^{\infty}(\dd_0 + \dd_1 + \dd_2 \dots)(Y_p) = \sum_{p=P}^{\infty}\dd(Y_p)= \dd(Y).
\]
It remains to construct such a $Y$.  We do so inductively.

As $X$ is a cocycle, the filtration on $\dd$ yields the equation
\begin{equation} 
\label{eq:xkd}
\dd_k(X_P) + \dd_1(X_{P+1}) \dots + \dd_0(X_{P+k}) = 0
\end{equation}
for all $k$.  In particular, $\dd_0(X_P) = 0$.  By Lemma \ref{lem:dd0}, there is a unique $Y_P$ satisfying $\dd_0(Y_P) = X_P$.  So suppose that there exists $Y_p$ such that equation (\ref{eq:inductive}) holds.  Lemma \ref{lem:filtered-d} states that
\begin{equation} 
\label{eq:dk=0}
(\dd_0 + \dd_1 + ... + \dd_k)^2 = 0
\end{equation}
for all $k$.  

Consider the equation
\begin{equation} 
\label{eq:ddinduct}
\dd_1(Y_p) + \dd_2(Y_{p}) + \dots + \dd_{p+1-P}(Y_P) + X_{p+1}.
\end{equation}
Applying $\dd_0$ to (\ref{eq:ddinduct}), then manipulating via (\ref{eq:dk=0}) and (\ref{eq:xkd}),
\begin{align}
-\dd_0\dd_1(Y_p) - \dd_0\dd_2(Y_{p-1}) &- \dots - \dd_0\dd_{p+1+P}(Y_P) + \dd_0X_{p+1} \\
&= -\dd_0\dd_0(Y_{p})-\dd_1(\dd_0 + \dd_1)(Y_{p})  \\ &\hspace{.4cm}- \dd_0(\dd_0 + \dd_1)(Y_{p-1}) - \dd_1(\dd_0 + \dd_1 + \dd_2)(Y_{p-1}) - \dd_ 2(\dd_0 + \dd_1 + \dd_2)(Y_{p-1}) \\&\hspace{.4cm}- \dots \\&\hspace{.4cm}- \dd_1(X_{p}) - \dots - \dd_{p+1}(X_P) \\
&= \dd_0(\dd_0(Y_{p}) + \dd_1(Y_{p-1}) + \dd_2(Y_{p-2}) + \dots + \dd_{p-P}(Y_P))
\\&\hspace{.4cm} + (\dd_0 + \dd_1)(\dd_0(Y_{p-1}) + \dd_1(Y_{p-2}) + \dots + \dd_{p-1-P}(Y_P))
 \\ &\hspace{.4cm} + \dots
 \\ &\hspace{.4cm} + (\dd_0 + \dots + \dd_{p+1-P})\dd_0(Y_P) \label{eq:statementy}
\end{align}
Use the induction hypothesis (\ref{eq:inductive}) to replace each statement about $Y$ in (\ref{eq:statementy}) with one about $X$:
\begin{align*}
&= \dd_0(X_{p}) + (\dd_0+\dd_1)(X_{p-1}) + \dots + (\dd_0 + \dots + \dd_{p+1-P})(X_P) \\
&= \sum_{k=P}^p \dd_0(X_k) + \dd_1(X_{k-1}) + \dots + \dd_k(X_P) \\
&= 0,
\end{align*}
where the final equality is equation (\ref{eq:xkd}).

By Lemma \ref{lem:dd0} there exists $Y_{p+1}$ satisfying 
\[
\dd_0(Y_{p+1}) = \dd_1(Y_p) + \dd_2(Y_{p}) + \dots + \dd_{p+1-P}(Y_P) + X_{p+1}.
\]
By construction, $\mu_{CZ}(\mathfrak{p}_*(Y_p)) = p$.  By Lemma \ref{lem:czaction}, the action of $Y_p$ grows like $p$.  In particular,
\[
\lim_{p\rightarrow\infty}\cg{A}(Y_p)\neq -\infty.
\]
Thus, $Y$ is indeed an element of $\lim\limits_{\substack{\rightarrow \\ a}}\lim\limits_{\substack{\leftarrow \\ m}} CF^*_{(a, \infty)}(K_m)$.  This proves the induction hypothesis, and so proves the proposition.

\end{proofprop6}

\subsection{Proof of Theorem \ref{thm:linebundle}}

Recall that the Reeb flow of a contact one-form is {\it Zoll} if it is periodic, and if all simple Reeb orbits have the same period.  We prove the following proposition.
\begin{proposition}
\label{prop:nested}
Let $\Sigma$ be a hypersurface of contact type with Zoll Reeb flow in a monotone symplectic manifold $E$.  Choose a sequence of positive numbers $\epsilon_1, \epsilon_2, \dots \rightarrow0$ such that 
\[
U_{\epsilon_i}\simeq(-\epsilon_i, \epsilon_i)\times\Sigma\subset E
\]
is a tubular neighborhood.  Suppose that for every $i$ there exists $\delta_i\in(-\epsilon_i, \epsilon_i)$ such that $\{\delta_i\}\times\Sigma$ contains a leaf-wise intersection point of some Hamiltonian $F$.  Then $\Sigma$ contains a leaf-wise intersection point of $F$ as well.
\end{proposition}

\begin{proof}
The proof is standard; we sketch the argument and refer to \cite{audin-d} for details.  Let $\{z_i\}$ be the set of leaf-wise intersection points in the statement of the proposition.  Without loss of generality, assume that the time support of $z_i$ occurs in the interval $[1/2, 1]$.  Fix embedded paths $\{y_i:[0, 1/2]\rightarrow\{\delta_i\}\times\Sigma\}$ that are flow lines of the vector field $\rho(t)\sigma_iR_{\alpha}$ for some $\sigma_i\in[0, 1)$ and satisfy $y_i(0) = z_i(1)$ and $y_i(1/2) = z_i(1/2)$.  Let $x_i = y_i\#z_i$ be the composition.  Thus, there is a Hamiltonian vector field $X_i$ defined by
\[
X_i = \left\{\begin{array}{cc} \rho(t)\sigma_iR_{\alpha} & t\in[0, 1/2] \\ X_F & t\in[1/2, 1]\end{array}\right.
\]
such that
\[
\dot{x}_i = X_i(x_i).
\]
As $supp(F)$ is compact, $y_i$ is embedded, and the Reeb flow is Zoll, the $L^2$ norm of $x_i$ is uniformly bounded and the family $\{x_i\}$ is equicontinuous.  There therefore exists a subsequence of $x_i$ with a $\cg{C}^0$ limit $x$ of $x_i$.  Abusing notation, we replace the sequence $x_i$ with the convergent subsequence.

Necessarily, there is a constant $\sigma\in[0, 1]$ such that
\[
\lim_{i\rightarrow\infty}\sigma_i = \sigma.
\]
Let
\[
X = \left\{\begin{array}{cc} \rho(t)\sigma R_{\alpha} & t\in[0, 1/2] \\ X_F & t\in[1/2, 1]\end{array}\right..
\]
Then
\begin{align*}
x(t) - x(0) - \int_0^tX(x(\tau))d\tau &= \lim_{i\rightarrow\infty}\left(x_i(t) - x_i(0) - \int_0^tX(x(\tau))d\tau\right) \\
&= \lim_{i\rightarrow\infty}\left(\int_0^t\dot{x}_i(\tau)d\tau - \int_0^tX(x(\tau))d\tau\right) \\
&= \lim_{i\rightarrow\infty}\left(\int_0^t\dot{x}_i(\tau) - X(x_i(\tau))d\tau\right) + \lim_{\rightarrow\infty}\left(\int_0^tX(x_i(\tau)) - X(x(\tau))d\tau\right) \\
&= \lim_{i\rightarrow\infty}\left(\int_0^t\dot{x}_i(\tau) - X_i(x_i(\tau)) + (\sigma_i - \sigma)\rho(t)R_{\alpha}(x_i)d\tau\right) + \lim_{i\rightarrow\infty}\left(\int_0^tX(x_i(\tau)) - X(x(\tau))d\tau\right) \\
&= 0.
\end{align*}
Boot-strapping, $x$ is $\cg{C}^{\infty}$ satisfying
\[
\dot{x} = X(x(t)),
\]
with $x(0) = x(1)\in \Sigma$.

The periodic orbits of $X$ correspond to leaf-wise intersection points of $F$; since $x(0)\in\Sigma$, $x$ corresponds to a leaf-wise intersection point of $F$ on $\Sigma$.

\end{proof}

We are finally ready to prove the Theorem.

\begin{proofthm1}
Let $U_{\epsilon} = (R-\epsilon, R+\epsilon)\times\Sigma$.  Suppose that there does not exist a circle subbundle inside $U_{\epsilon}$ with a leaf-wise intersection point of $F$.  Set $P_1 = R-\epsilon$.  Stringing together Propositions \ref{prop:qi(1)}, \ref{prop:qi(2)}, \ref{prop:qi(3)}, \ref{prop:qi(4)}, and \ref{prop:qi(5)}, there is an isomorphism
\[
\lim_{\substack{\rightarrow \\ a}}\lim_{\substack{\leftarrow \\ m}}HF^*_{(a, \infty)}(K_{m}) \simeq \lim_{\substack{\rightarrow \\ a}}\lim_{\substack{\rightarrow \\ n}}HF^*_{(a, \infty)}(H_{n})
\]
There is a surjective map
\[
H\left(\lim\limits_{\substack{\rightarrow \\ a}}\lim\limits_{\substack{\leftarrow \\ m}} CF^*_{(a, \infty)}(K_m) \right) \rightarrow \lim\limits_{\substack{\rightarrow \\ a}}\lim\limits_{\substack{\leftarrow \\ m}} HF^*_{(a, \infty)}(K_m).
\]
The left-hand side is isomorphic to $\widehat{SH_*}(K)$, which, by Proposition \ref{prop:shvanishes} is zero:
\[
\widehat{SH_*}(K) = 0.
\]
Thus, 
\[
\lim_{\substack{\rightarrow \\ a}}\lim_{\substack{\leftarrow \\ m}}HF^*_{(a, \infty)}(K_{m}) = 0.
\]
The right-hand side is $SH^*(H)$, which, by Lemma \ref{lem:truncisusual}, is non-vanishing under the assumption that symplectic cohomology is non-vanishing.

We reach a contradiction and conclude that there exists a circle subbundle in $U_{\epsilon}$ with a leaf-wise intersection point of $F$.  Let $\epsilon_1, \epsilon_2, \dots\rightarrow0$ be a sequence defining a nested sequence of neighborhoods $U_{\epsilon_i}$.  Each $U_{\epsilon_i}$ contains a circle subbundle with a leaf-wise intersection point.  By Proposition \ref{prop:nested}, $\Sigma = \{R\}\times\Sigma$ contains a leaf-wise intersection point.

\end{proofthm1}

Examples of negative monotone line bundles with non-vanishing symplectic cohomology were found by Ritter \cite{ritter-gromov}\cite{ritter-fano}.  He showed that the symplectic cohomology is non-vanishing whenever the base $M$ is toric.  
\addtocounter{corollary}{-5}
\begin{corollary}
Let $E \rightarrow M$ be a monotone negative line bundle with negativity constant $k$ and monotonicity constant $\kappa$ over a  toric symplectic manifold.  For any compactly-supported Hamiltonian, the radius-$\frac{1}{\sqrt{k\kappa\pi}}$ circle subbundle  contains a leaf-wise intersection point.
\end{corollary}
\addtocounter{corollary}{+4}

\section{Other line bundles}
\label{sec:thm2}
Monotone line bundles have exactly one ``special'' circle subbundle; the hypersurface for which Theorem \ref{thm:linebundle} produces a leaf-wise intersection point.  This is the circle subbundle on which Floer-essential Lagrangian tori appear (in toric line bundles) and it is the only circle subbundle on which a Rabinowitz-Floer-type invariant $RFH^*(\Sigma)$ is non-zero \cite{ritter-smith}\cite{venkatesh-neg}.

Non-monotone line bundles can have multiple special subbundles.  Although very little is understood about their Floer-essential Lagrangian tori, we showed in \cite{venkatesh-neg} that $RFH^*(\Sigma)$ is non-vanishing at a finite number of circle subbundles in {\it weak\textsuperscript{+} monotone} line bundles.  Precisely, the cohomological class $\mathfrak{p}^*c_1^E$ acts on the quantum cohomology of $E$, $QH^*(E)$, by quantum cup product.  This action has a finite set of eigenvalues $\{\lambda_1, \dots, \lambda_n\}$.  

\begin{theorem}[Venkatesh \cite{venkatesh-neg}]
\label{thm:venkneg}
The invariant $RFH^*(\Sigma_R)$ of a circle subbundle of radius $R$ in a weak\textsuperscript{+}-monotone toric negative line bundle is non-vanishing precisely when
\[
R = \sqrt{\frac{ev(\lambda_i)}{k\pi}}
\]
for some $\lambda_i\in\{\lambda_1, \dots, \lambda_n\}$.
\end{theorem}

We conjecture that each of the hypersurfaces indicated in the Theorem contain a leaf-wise intersection point for any compactly-supported Hamiltonian function.
\begin{conjecture}
Let $E$ be a negative line bundle, and let $F:E\times\RR\rightarrow\RR$ be a compactly-supported Hamiltonian function.  A circle subbundle $\Sigma_R$ of radius $R$ contains a leaf-wise intersection point of $F$ if
\[
R = \sqrt{\frac{ev(\lambda_i)}{k\pi}}.
\]
\end{conjecture}

The invariant in Theorem \ref{thm:venkneg} builds upon a completion of the symplectic cochain complex $SC^*(H)$.  Recall from equation (\ref{eq:action}) that $SC^*(H)$ has a valuation by action, denoted by $\cg{A}:SC^*(H)\rightarrow\RR$.  This valuation defines a non-Archimedean norm, via
\[
||X|| = e^{-\cg{A}(X)}.
\]
Denote by $\dd$ the differential on $SC^*(H)$.  Let
\[
\widehat{\ker(\dd)}
\]
be the completion of $\ker(\dd)$ with respect to the norm $||\cdot||$, and let
\[
\overline{\im(\dd)}
\]
be the closure of $\im(\dd)\subset\widehat{\ker(\dd)}$.
Define {\it completed symplectic cohomology} by
\[
\widehat{SH^*}(H) := \bigslant{\widehat{\ker(\dd)}}{\overline{\im(\dd)}}.
\]
A preliminary result from \cite{venkatesh-neg} shows that
\begin{equation}
\label{eq:shbehavior}
\widehat{SH^*}(H) \left\{\begin{array}{cc} = 0 & R < \min\limits_{\lambda\in\{\lambda_1, \dots, \lambda_n\}}\sqrt{\frac{ev(\lambda)}{k\pi}} \\
\neq 0 & \text{else}\end{array}\right..
\end{equation}

The proof of Theorem \ref{thm:linebundle} relies on the monotonicity of the line bundle to prove technical results on action bounds.  Replacing monotonicity with {\it weak\textsuperscript{+} monotonicity}, our techniques applied to the result (\ref{eq:shbehavior}) yield a less precise result.

\addtocounter{theorem}{-5}
\begin{theorem}
Let $E$ be a weak\textsuperscript{+}-monotone negative line bundle.  Let $\lambda$ be a largest non-zero eigenvalue of the action of $\mathfrak{p}^*c_1^E$ by quantum cup product on $QH^*(E)$.  For any compactly-supported Hamiltonian, there exists a circle subbundle $\Sigma_r$ with $r \in \left(\sqrt{\frac{ev(\lambda)}{k\pi}}, \infty\right)$ such that $\Sigma_r$ has a leaf-wise intersection point. 
\end{theorem}

\begin{proof}
As in the proof of Proposition \ref{prop:qi(2)}, there is a constant $C$, independent of $a$, that produces chain maps
\[
SC^*_{(a, \infty)}(H) \rightarrow SC^*_{(a - C, \infty)}(H^F)\rightarrow SC^*_{(a - 2C, \infty)}(H)
\]
such that the induced map on homology
\[
SH^*_{(a, \infty)}(H) \xrightarrow{\simeq} SH^*_{(a - 2C, \infty)}(H)
\]
is the canonical inclusion.  It follows that there are maps on the completions
\[
\widehat{SH^*}(H) \rightarrow \widehat{SH^*}(H^F)\rightarrow \widehat{SH^*}(H)
\]
such that the composition is an isomorphism.

Similarly, there are maps
\[
\widehat{SH^*}(H^F)\rightarrow \widehat{SH^*}(H)\rightarrow\widehat{SH^*}(H^F)
\]
whose composition is an isomorphism.  We conclude that there is an isomorphism
\begin{equation}
\label{eq:shf1}
\widehat{SH^*}(H^F) \xrightarrow{\simeq} \widehat{SH^*}(H).
\end{equation}

Choose $R < \min\limits_{\lambda\in\{\lambda_1, .\dots, \lambda_n\}}\sqrt{\frac{ev(\lambda)}{k\pi}}$.  By (\ref{eq:shbehavior}),
\begin{equation}
\label{eq:shf2}
\widehat{SH^*}(H) = 0.
\end{equation}
On the other hand, Ritter showed in \cite{ritter-gromov} that if $\lambda\neq 0$, 
\begin{equation}
\label{eq:shf3}
SH^*(H)\neq 0.
\end{equation}
Standard results in Floer theory yield an isomorphism
\begin{equation}
\label{eq:shf4}
SH^*(H^F)\simeq SH^*(H).
\end{equation}
If there does not exist a circle subbundle $\Sigma_r$, $r\in \left(\sqrt{\frac{ev(\lambda)}{k\pi}}, \infty\right)$, with a leaf-wise intersection point, then $\cg{P}(H^F)$ is a finite set, and $SC^*(H^F)$ is finitely generated as a vector space.  In particular, $\ker(\dd)$ is finitely generated, and is thus already complete.  Likewise, $\im(\dd)$ is finitely generated and is thus already closed.  There is therefore an isomorphism
\begin{equation}
\label{eq:shf5}
\widehat{SH^*}(H^F)\xrightarrow{\simeq} SH^*(H^F).
\end{equation}
Stringing together the equations (\ref{eq:shf1}) -- (\ref{eq:shf5}), 
\[
0\neq SH^*(H) \simeq SH^*(H^F)\simeq \widehat{SH^*}(H^F)\simeq \widehat{SH^*}(H) = 0,
\]
and we reach a contradiction.  We conclude that there exists a circle subbundle $\Sigma_r$ with a leaf-wise intersection point.

\end{proof}

\section{Leaf-wise intersection points from Lagrangians}
\label{sec:lagrangian}

The circle subbundles studied in this paper are conjecturally distinguished by Floer-essential Lagrangian submanifolds living inside them.  Such submanifolds point to a variety of fixed-point results, including results about leaf-wise intersection points.  In this section, we prove a theorem about leaf-wise intersection points on more general monotone symplectic manifolds.

\addtocounter{theorem}{+4}
\begin{theorem}
\label{thm:lagrangians}
Let $(E, \Omega)$ be a monotone symplectic manifold that is either compact or convex at infinity.  Let $\Sigma\subset E$ be a hypersurface of contact type containing a compact Lagrangian $L$, and let $(-\epsilon, \epsilon)\times \Sigma$ be a Weinstein neighborhood of $\Sigma$.  If there exists a local system $\gamma$ such that the Lagrangian Floer homology $HF^*(L, \gamma)$ is non-vanishing, then for any Hamiltonian perturbation, there exists $r\in(-\epsilon, \epsilon)$ such that $\{r\}\times\Sigma$ contains a leaf-wise intersection point.  If the Reeb flow of $\Sigma$ is Zoll, $\Sigma$ itself contains a leaf-wise intersection point.
\end{theorem}

The proof of Theorem \ref{thm:lagrangians} follows a well-documented algorithm.  Similarly to the recipes of Sections \ref{sec:thm1} and \ref{sec:thm2}
there exists a Floer invariant, denoted by $\widehat{SH^*}(\Sigma, E)$, that encodes leaf-wise intersection points close to $\Sigma$.    Precisely, if $\widehat{SH^*}(\Sigma, E)\neq 0$ then there exists a contact hypersurface arbitrarily close to $\Sigma$ with a leaf-wise intersection point.  If the Reeb flow of $\Sigma$ is Zoll, then this encoding, combined with Proposition \ref{prop:nested}, shows that $\Sigma$ itself has a leaf-wise intersection point.  To prove the Theorem, it thus suffices to show that $\widehat{SH^*}(\Sigma, E)\neq 0$.

Our method for proving that $\widehat{SH^*}(\Sigma, E)\neq 0$ originated in \cite{seidel-biased} and was adapted to the current framework in our paper \cite{venkatesh-rab}.  There is a map
\[
\widehat{SH^*}(\Sigma, E)\rightarrow HF^*(L, \gamma)
\]
that fits into a commutative diagram with the quantum cohomology of $E$, denoted by $QH^*(E)$.
\begin{equation}
\label{fig:nonvanish}
\begin{tikzcd}
QH^*(E) \arrow{r} \arrow{rd} & \widehat{SH^*}(\Sigma, E) \arrow{d} \\
& HF^*(L, \gamma)
\end{tikzcd}
\end{equation}
$HF^*(L, \gamma)$ is a ring with unit, and the image of the map
\[
QH^*(E)\rightarrow HF^*(L, \gamma)
\]
always contains the unit.  If $HF^*(L, \gamma)\neq 0$, then, in particular, the image of this map is non-zero.  It follows that 
\[
\widehat{SH^*}(\Sigma, E)\neq 0
\]
as well.

The cohomology theory $\widehat{SH^*}(\Sigma, E)$ relies on a new family of Hamiltonians, defined as follows.  A contact hypersurface $(\Sigma, \alpha)$ has a family of nested Weinstein neighborhood
\[
U_1 \supset U_2 \supset U_3\supset \dots \supset \Sigma
\]
whose intersection is $\Sigma$
\[
\bigcap_i U_i = \Sigma.
\]
We assume without loss of generality that there are positive real numbers $\epsilon_1 > \epsilon_2 > \epsilon_3 \dots > 0$ with symplectomorphisms
\[
U_i\simeq(-\epsilon_i, \epsilon_i)_r\times\Sigma,
\]
when the latter is equipped with the symplectic form $d(e^r\alpha)$.  Fix $i\in\NN$.  Let $h_n^i:\RR_x\rightarrow\RR$ be a smoothing of the function that is
\begin{enumerate}
\item equal to $-nx$ on $(-\epsilon_i, 0]$,
\item equal to $nx$ on $[0, \epsilon_i)$,
\item and constant everywhere else.
\end{enumerate}
See Figure \ref{fig:v}.

\begin{figure}[!htbp]
\centering
\begin{tikzpicture}[scale=5]
\draw  (.5,-1.01916631) -- (.5,-.07) node [right = 1pt] {$\RR$};
\draw(0,-1.01916631)-- (1, -1.03660231)node [right = 1pt] {$r$};
\draw[violet] (0,-.07)-- (.267948, -.07);
\draw[violet] (.267948, -.07) -- (.5, -1.01916631);
\draw[violet]  (0.732052,-.07) -- (.5, -1.01916631);
\draw[violet] (0.732052,-.07)-- (1, -.07) node [right=1pt] {$h_n^i$};
\draw (.267948, -1.03660231)node [anchor=north]{{\tiny$-\epsilon_i$}};
\draw (0.732052, -1.03660231)node [anchor=north]{{\tiny$\epsilon_i$}};
\draw (.267948, -1.02660231) -- (.267948, -1.04660231);
\draw (0.732052, -1.02660231) -- (0.732052, -1.04660231);
\end{tikzpicture}
\caption{The hamiltonians $h_n^i$}
\label{fig:v}
\end{figure}
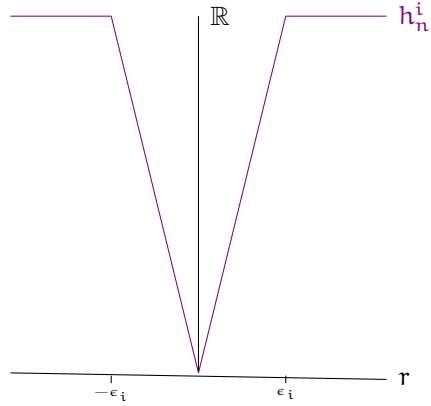

As in Subsection \ref{subsec:trunc}, let $\rho(t)$ be a bump function with support on $[0, 1/2]$ that integrates to 1, and let $F$ be a compactly-supported Hamiltonian with time-support in $[1/2, 1]$.  Define
\[
H_n^i = \rho(t)h_n^i(e^r) + F.
\]
Define quotient complexes
\[
CF^*_a(H_n^i) := \bigslant{CF^*(H_n^i)}{CF^*_{(a, \infty)}(H_n^i)}.
\]
Choose monotone-increasing homotopies $CF^*_a(H_n^i) \rightarrow CF^*_a(H_{n+1}^i)$.  With these choices, the family $\{CF^*_a(H_n^i)\}$ is bidirected; once by $n\in\NN_{>0}$ and once by $a\in\RR$.

Define the {\it reduced symplectic cohomology} of $U_i$, written $\widehat{SH^*}(U_i)$, to be the cohomology of the cochain complex
\[
\widehat{SC^*}(U_i) = \lim_{\substack{\leftarrow \\ a}}\lim_{\substack{\rightarrow \\ n}}CF^*_{a}(H_n^i).
\]

As with Lemma \ref{lem:perturbF}
we may assume, without changing the leaf-wise intersection points of $F$, that the periodic orbits of each $H_n^i$, and of $F$ itself, are non-degenerate.

\subsection{Lagrangian quantum cohomology}
In this subsection we recall a model of $HF^*(L, \gamma)$ called {\it Lagrangian quantum cohomology} and we define and study the maps in Figure \ref{fig:nonvanish}.  This picture was constructed in an analogous setup in \cite{venkatesh-rab}.
Rather than repeat each argument, we recall the results detailed there
and argue that we can transfer each argument to this setting.  We refer to \cite{venkatesh-rab}
for a detailed exposition.

Choose a Morse function $f:L\rightarrow\RR$ and a metric $g$ on $L$.  Denote by $\Phi_t$ the gradient flow of $-\nabla_g(f)$.  Choose critical points $x$ and $y$ of $f$ and a relative homology class $\beta\in H^2(E, L)$, and denote by $\cg{M}_{\ell}(f, g, x, y, \beta)$ the moduli space of tuples $(u_1, \dots, u_{\ell})$, where
\begin{enumerate}
\item $u_i:(D^2, \dd D^2)\rightarrow (E, L)$ is a non-constant $J$-holomorphic disk,
\item there exists $t\in(-\infty, 0)$ such that $u_i(-1) = \Phi_t(u_{i+1}(1))$ whenever $i\leq 1 < \ell$,
\item $u_1(1)$ is in the unstable manifold of $x$ and $u_{\ell}(-1)$ is in the stable manifold of $y$,
\item and $[u_1\#u_2\#\dots\#u_{\ell}]\in\beta$.
\end{enumerate}
Let $Aut(D^2, \pm 1)$ be the automorphisms of the disk fixing $-1$ and $1$, so that $Aut(D^2, \pm 1)$ acts on $\cg{M}_{\ell}(f, g, x, y, \beta)$.  Denote by $\cg{M}_0(f, g, x, y)$ the space of Morse trajectories from $x$ to $y$.  Note that $\RR$ acts canonically on $\cg{M}_0(f, g, x, y)$.  Define $\cg{M}^0(f, g, x, y, \beta)$ to be the rigid elements of 
\[
\bigslant{\cg{M}_{\ell}(f, g, x, y, \beta)}{Aut(D^2, \pm 1)}\cup \bigslant{\cg{M}_0(f, g, x, y)}{\RR}.
\]
We call the elements of $\cg{M}^0(f, g, x, y)$ {\it pearly trajectories}.

Choose a local system $\gamma$ on $L$, and note that there is a pairing $\la \gamma, \sigma\ra$ for each $\sigma\in H_1(L)$.  Let $CF^*(L, \gamma)$ be the cochain complex that is generated by critical points of $f$
\[
CF^*(L, \gamma) = \bigoplus_{x\in Crit(f)}\Lambda\la x\ra
\]
and whose differential is given on critical points of $f$ by
\[
\dd(x) = \sum_{\substack{y\in Crit(f) \\ \ell\in\NN \\ \beta\in H^2(M, L)}}\sum_{{\bf u}\in\cg{M}^0(f, g, x, y)}T^{-\Omega(\beta)}\la \gamma, [\dd u_1\#\dd u_2\#\dots \#\dd u_{\ell}] \rangle y.
\]
Lagrangian Floer cohomology, denoted by $HF^*(L, \gamma)$, is the cohomology of $CF^*(L, \gamma)$.

Analogously to symplectic cohomology, Lagrangian Floer cohomology has action truncated versions $HF^*_{(a, \infty)}(L; \gamma)$ and $HF^*_a(L; \gamma)$.

\begin{lemma}[Lemma 6 in \cite{venkatesh-rab}]
There is a map
\[
\widehat{SH^*}(U_i)\rightarrow HF^*(L, \gamma).
\]
\end{lemma}

\begin{proof}
Lemma 6 in \cite{venkatesh-rab} and the discussion following the proof of Lemma 6 consider autonomous Hamiltonians, which we denote here by $\{H_n'\}$, on a domain with convex boundary and produce a map
\[
\lim_{\substack{\leftarrow \\ a}}\lim_{\substack{\rightarrow \\ n}}HF^*_a(H_n') \rightarrow HF^*_a(L; \gamma)
\]
for any $n$ and $a$.  This map counts ``half-cylinders'', amalgamated with pearly trajectories.

The only facts used about the family $\{H_n'\}$ to construct this map are that
\begin{enumerate}
\item $H_n \leq 0$ on a small neighborhood of $L$,
\item and the continuation maps $H_n \rightarrow H_{n+1}$ are monotone-increasing.
\end{enumerate}
These can be easily arranged for our family $\{H_n\}$ without changing the Floer complexes by shifting each function $h_n$ and the Hamiltonian $F$ down by a constant.  After these adjustments, the proof translates verbatim.

\end{proof}
\begin{lemma}[Lemma 7 in \cite{venkatesh-rab}]

The map
\[
\Phi: \widehat{SH^*}(U_i)\rightarrow HF^*(L, \gamma).
\]
is non-vanishing.
\end{lemma}

\begin{proof}
The proof of Lemma 7 in \cite{venkatesh-rab} translates verbatim.  We briefly sketch it here.  By construction, the diagram
\begin{equation}
\begin{tikzcd}
HF^*(H_0) \arrow{r}{\iota} \arrow{dr}{\phi} & \widehat{SH^*}(\Sigma, E) \arrow{d}{\Phi} \\
& HF^*(L; \gamma)
\end{tikzcd}
\end{equation}
commutes.  It thus suffices to show that the map
\[
\phi: HF^*(H_0) \rightarrow HF^*(L; \gamma)
\]
is non-zero.

Assume without loss of generality that $f$ has a unique minimum $\mathfrak{m}$.  As $HF^*(L, \gamma)$ is a unital ring, $\mathfrak{m}$ represents the unit, and as $HF^*(L; \gamma)$ is non-zero by assumption, $[\mathfrak{m}]\neq 0$ as well.  Denote by $\cg{M}_0(\mathfrak{m})$ the moduli space of all pearly trajectories $(u_1, \dots, u_{\ell})$ with $u_{\ell}$ now possibly constant, $u_{\ell}(-1)$ in the stable manifold of $\mathfrak{m}$, and no restrictions on $u_1(1)$.  Because $\mathfrak{m}$ is a minimum, there is a single rigid element in $\cg{M}_0(\mathfrak{m})$: the constant map $(u_{\ell}:D^2\rightarrow\mathfrak{m})$.

Let $H_s$ be a family of Hamiltonians that is identically zero when $s << 0$ and is equal to $H_0$ when $s >> 0$.  Choose a generic family $J_s$ of $\Omega$-compatible almost-complex structures.  Let $\cg{M}^0(x)$ be the moduli space of rigid trajectories $w:\RR\times S^1\rightarrow E$ that satisfy the Floer equation for $(H_s, J_s)$ and with $\lim\limits_{s\rightarrow\infty}w(s, t) = x(t)$.  Define
\[
Z = \sum_{\substack{x\in \cg{P}(H_0) \\ w\in\cg{M}^0(x)}} T^{\Omega([\tilde{x}\#w])} x.
\]
Concatenating Floer cylinders, $\phi$ takes $[Z]$ to a cohomological-level sum of the elements in $\cg{M}_0(\mathfrak{m})$.  Thus,
\[
\phi([Z]) = [\mathfrak{m}] \neq 0.
\]

\end{proof}

\begin{corollary}
\label{cor:nonvanish}
The cohomology theory $\widehat{SH^*}(U_i)$ is non-zero.
\end{corollary}

\subsection{Proof of Theorem \ref{thm:lagrangians}}

The proof of Theorem \ref{thm:lagrangians} follows almost immediately from the previous section.

\begin{proofthm8}
Choose a neighborhood $U_i\simeq (-\epsilon_i, \epsilon_i)\times\Sigma$ of $\Sigma$ and assume for contradiction that there does not exist a hypersurface $\{r\}\times\Sigma\subset U_i$ with a leaf-wise intersection point of $F$.  Then all elements of 
\[
\cg{P}(\cg{H}^i) := \bigcup_n \cg{P}(H_n^i)
\]
correspond to periodic orbits of $F$ with starting point in $E\setminus U_i$.  By assumption, the periodic orbits of $F$ are non-degenerate, and therefore finite.  Let $\{\tilde{x}_1, \dots, \tilde{x}_{\ell}\}$ be the elements of $\cg{P}(\cg{H}^i)$ equipped with fixed cappings.  Thus, there exists a constant $C_{\mu} > 0$ with
\[
|\mu_{CZ}(\tilde{x}_j)| < C_{\mu}
\]
for all $1\leq j\leq\ell$.

Similarly, there exists a constant $C_F$ with
\[
|\cg{A}_F(\tilde{x}_j)| < C_F.
\]
By construction, $H_n^i(x(0)) \simeq n\cdot\epsilon_i$.  Therefore, adjusting $C_F$ slightly if necessary,
\[
\cg{A}_{H_n^i}(\tilde{x}_j) > n\cdot\epsilon_i - C_F
\]
for all $n, j$.

Consider a graded component $\widehat{SH^{\ell}}(U_i)$.  By the monotonicity of $E$, there exists a constant $\kappa > 0$ with
\[
c_1^{TE} = \kappa[\Omega]
\]
on elements of the image of $\pi_2(E)\hookrightarrow H^2(E)$.  Let $\lambda\in\Lambda$.  Then
\[
|\lambda x_j| = {\ell}
\]  
implies that 
\[
\kappa ev(\lambda) > {\ell} - C_{\mu},
\]
or
\[
ev(\lambda) > \frac{{\ell}-C_{\mu}}{\kappa}.
\]
It follows that
\[
\cg{A}_{H_n^i}(\lambda\tilde{x}_j) > n\cdot\epsilon_i - C_F + \frac{{\ell}-C_{\mu}}{\kappa}.
\]
Fix $a >> 0$.  For large enough $n$,
\[
n\cdot\epsilon_i - C_F + \frac{{\ell}-C_{\mu}}{\kappa} > a
\]
for all $j$, and so
\[
\lim_{\substack{\rightarrow \\ n}}HF^{\ell}_a(H_n^i) = 0.
\]
Taking the inverse limit over $a\rightarrow\infty$, we conclude that
\[
\widehat{SH^*}(U_i) = 0,
\]
contradicting Corollary \ref{cor:nonvanish}.  We conclude that there exists a contact hypersurface $\{r\}\times\Sigma\subset U_i$ with a leaf-wise intersection point of $F$.  Taking the limit as $i\rightarrow\infty$, we conclude the first statement of the theorem.

The second statement now follows immediately from Proposition \ref{prop:nested}.

\end{proofthm8}

\appendix
\section{Hamiltonians with smaller time-support}

Let $\cg{H}_n$ be a Hamiltonian as in Subsection \ref{subsec:trunc}.  Thus, there is a Hamiltonian $H_n$ of the form
\[
H_n = h_n(\kappa\pi r^2) + (1+\pi r^2)\mathfrak{p}^*G,
\]
where $G:B\rightarrow\RR$ is a Morse function and $h_n(\kappa\pi r^2) = n\kappa\pi r^2$ outside of a compact set $D$; there is a choice of bump function $\rho:\RR\rightarrow\RR$ integrating to one, with
\[
supp(\rho)\subset[0, 1/2];
\]
and
\[
\cg{H}_n = \rho(t)\cdot H_n.
\]
The proof of the following Lemma also proves the integrated maximum principle in our setting.
\begin{lemma}
\label{lem:maxprinc}
For generic $\Omega$-tame $J$, conical outside of $D$, the symplectic cochain complex $SC^*(\cg{H}_n)$ is well-defined.
\end{lemma}

\begin{proof}
It suffices to show that the Floer solutions contributing to both the Floer differential and the continuation maps remain in a compact set.  We use the integrated maximum principal of \cite{abouzaid}.  We prove that Floer solutions contributing to the Floer differential remain in a compact set; the case of continuation maps proceeds analogously.  Let $u:\RR\times S^1\rightarrow E$ be a Floer solution for $\cg{H}_n$.  Choose a generic circle subbundle $\mathfrak{S}$ of radius $\sigma$ outside of $D$:
\[
D\cap\mathfrak{S} = \emptyset,
\]
where genericity means that $\Sigma$ intersects $\mathfrak{S}$ transversally.  Let $\cg{D}$ be the region bounded by $\mathfrak{S}$, and denote $\Sigma = u^{-1}(E\setminus\cg{D})$.  Assume for contradiction that $\Sigma$ is non-empty.  Let $v\colon \Sigma\longrightarrow E\setminus\cg{D}$ be the restriction of $u$.  We will equate $\Sigma$ with its image under the inclusion into $\RR\times S^1$ and use the coordinates $(s, t)$ induced on $\Int(\Sigma)$.

Let $c_x = h_n(k\pi x^2) - h'_n(k\pi x^2)k\pi x^2$ be the $y$-intercept of the tangent line to $h_n(k\pi r^2)$ at $k\pi x^2$.  Denote by $X_G^h$ the unique lift to $\mathfrak{S}$ of the vector field $X_G$ on $B$ that is induced by the connection form $\alpha$.  Then on $\mathfrak{S}$, 
\[
\cg{H}_n = \rho(t)\cdot h_n'(k\pi \sigma^2)k\pi \sigma^2 + \rho(t)\cdot(1 + k\pi \sigma^2)\mathfrak{p}^*G + \rho(t)\cdot c_{\sigma}
\]
 (and $(X_{\cg{H}_n})_t = \rho(t)\cdot(h_n'(k\pi \sigma^2) + \mathfrak{p}^*G)R_{\alpha} + \rho(t)\cdot X_G^h$).  In particular, 
\[
\cg{H}_n\big|_{\mathfrak{S}} = (1 + k\pi \sigma^2)\alpha(X_{\cg{H}_n}) - \rho(t)\big(h_n'(k\pi \sigma^2) + c_{\sigma}\big).
\]

Let $\dd\Sigma$ be the union of the boundary components of $\Sigma$ mapping into $\mathfrak{S}$.  Note that we have chosen $J$ to be $\Omega$-tame, so that
\[
E_J(v) := \frac{1}{2}\int_{\Sigma}\left(\Omega(\dd_sv, J\dd_sv\right) + \Omega\left(\dd_tv - X_{\cg{H}_n}(v), J(\dd_tv - X_{\cg{H}_n}(v))\right)ds\wedge dt \geq 0
\]
Shuffling terms, we have
\begin{equation*}
\begin{split}
E_J(v) &= \int_{\Sigma} v^*\Omega - v^*d\cg{H}_n\otimes dt \\
&= \int_{\dd\Sigma_+} (1 + k\pi r^2)v^*\alpha - \cg{H}_n(v(s, t))dt
 \\&= \int_{\dd\Sigma_+} (1 + k\pi \sigma^2)v^*\alpha - (1 + k\pi \sigma^2)\alpha(X_{\cg{H}_n})\otimes dt + \rho(t)\cdot\left(h'_n(k\pi \sigma^2)- c_{\sigma}\right)dt.
\end{split}
\end{equation*}
A solution $v(s, t)$ of Floer's equation satisfies 
\[
(dv - X_{\cg{H}_n}\otimes dt)^{(0, 1)} = 0.
\] 
The conical condition on $J$ says that, on $\mathfrak{S}$, $JR_{\alpha}$ is proportional to $\dd_r$ and $JX_F^h$ lives in the horizontal distribution.  So altogether, $\alpha(JX_{\cg{H}_n}) = 0$.  Thus, 
\begin{align*}
\int_{\dd\Sigma_+}(1 + k\pi \sigma^2)\alpha(dv - X_{\cg{H}_n}\otimes dt) + &\rho(t)\left(h'_n(k\pi \sigma^2)- c_{\sigma}\right)dt  \\
&= \int_{\dd\Sigma_+} -(1 + k\pi \sigma^2)\alpha\circ J(dv - X_{\cg{H}_n}\otimes dt)\circ j + \rho(t)\left(h'_n(k\pi \sigma^2)- c_{\sigma}\right)dt  \\
&= \int_{\dd\Sigma_+} -\frac{1 + k\pi \sigma^2}{2k\pi\sigma}dr \circ dv \circ j + \rho(t)\left(h'_n(k\pi \sigma^2)- c_{\sigma}\right)dt  \\
&\leq \int_{\dd\Sigma_+}\rho(t)\left(h'_n(k\pi \sigma^2)- c_{\sigma}\right)dt, 
\end{align*}
where the last inequality follows because a vector $\zeta$ that is positively-oriented with respect to the boundary orientation satisfies $dr\circ j(\zeta) \geq 0$.

$\Sigma$ is a collection of bounded regions in $\CC^*$ on which $\rho(t)\left(h'_n(k\pi \sigma^2)- c_{\sigma}\right)dt $ is exact:
\[
\int_{\dd\Sigma_+}\rho(t)\left(h'_n(k\pi \sigma^2)- c_{\sigma}\right)dt = 0.
\]
Thus,
\[
E_J(v) \leq 0,
\]
implying that $v$ is constant.  But $\Sigma$ is non-empty and intersects $\mathfrak{S}$ transversally; and we reach a contradiction.

\end{proof}

\begin{lemma}
\label{lem:truncisusual}
If symplectic cohomology is non-vanishing, the cohomology theory $SH^*(\cg{H})$ is also non-vanishing:
\[
SH^*(\cg{H})\neq 0.
\]
\end{lemma}

\begin{proof}
Choose any function $\rho:\RR\times\RR\rightarrow \RR$ satisfying
\[
\rho(s, t) = \rho(t) \hspace{1cm} s >> 0
\]
and
\[
\rho(s, t) = \max_t\rho(t)  \hspace{1cm} s << 0 \text{ and } t\in[0, 1].
\]
For each $n\in\NN$, define a family of Hamiltonians $H_n^s$ by
\[
H_n^s = \rho(s, t)H_n.
\]
Then $\{H_n^s\}_{n\in\NN}$ induces a continuation map
\[
SC^*(\cg{H})\rightarrow SC^*(\max_t\rho(t) \cdot H),
\]
which is well-defined by the methods described in the proof of Lemma \ref{lem:maxprinc}.  Standard methods imply that this induces a map of unital rings on the level of homology
\[
SH^*(\mathcal{H})\rightarrow SH^*(\max_t\rho(t) \cdot H).
\]
As $SH^*(\max_t\rho(t) \cdot H)$ is non-zero, the unit in $SH^*(\max_t\rho(t) \cdot H)$ is non-zero as well.  It follows that $SH^*(\mathcal{H})$ has a non-zero unit, and so must be non-zero.

\end{proof}

\bibliography{biblio}{}
\bibliographystyle{plain}

\end{document}